\documentclass[11pt]{article}
\usepackage{geometry}
\usepackage{color}
\usepackage{float}
\usepackage{textcomp}
\usepackage{url}
\usepackage{amsthm}
\usepackage{amsmath}
\usepackage{amssymb}%
\usepackage{mdframed}
\usepackage{multirow}
\usepackage{changes}
\usepackage{hyperref}

\usepackage{mathtools}
\usepackage{booktabs}    
\usepackage{subfigure} 
\usepackage{caption} 
\usepackage{multirow}
\usepackage{multicol}    
\usepackage{array}       
\usepackage{graphicx}
\graphicspath{{figuras/}}
\usepackage{empheq,amsmath}

\usepackage
[ 
lined, 
boxruled, 
commentsnumbered
]
{algorithm2e}
\DecMargin{0.1cm} 

\newtheorem{theorem}{Theorem}[section]
\newtheorem{lemma}[theorem]{Lemma}
\newtheorem{corollary}[theorem]{Corollary}

\newtheorem{remark}[theorem]{Remark}
\newtheorem{definition}[theorem]{Definition}

\newcommand{\bi}{\begin{itemize}}
\newcommand{\ei}{\end{itemize}}
\newcommand{\ba}{\begin{array}}
\newcommand{\ea}{\end{array}}

\begin{document}

\title{\textbf{TRFD: A derivative-free trust-region method based on finite differences for composite nonsmooth optimization}}


\author{D\^an\^a Davar\thanks{Department of Mathematical Engineering, ICTEAM Institute, Université catholique de Louvain, B-1348 Louvain-la-Neuve, Belgium (dana.davar@uclouvain.be). This author was supported by the French Community of Belgium (FSR program).} 
\and Geovani N. Grapiglia\thanks{Department of Mathematical Engineering, ICTEAM Institute, Université catholique de Louvain, B-1348 Louvain-la-Neuve, Belgium (geovani.grapiglia@uclouvain.be). This author was partially supported by FRS-FNRS, Belgium (Grant CDR J.0081.23).} }

\date{October 11, 2024}

\maketitle

\begin{abstract}
In this work we present TRFD, a derivative-free trust-region method based on finite differences for minimizing composite functions of the form $f(x)=h(F(x))$, where $F$ is a black-box function assumed to have a Lipschitz continuous Jacobian, and $h$ is a known convex Lipschitz function, possibly nonsmooth. The method approximates the Jacobian of $F$ via forward finite differences. We establish an upper bound for the number of evaluations of $F$ that TRFD requires to find an $\epsilon$-approximate stationary point. For L1 and Minimax problems, we show that our complexity bound reduces to $\mathcal{O}(n\epsilon^{-2})$ for specific instances of TRFD, where $n$ is the number of variables of the problem. Assuming that $h$ is monotone and that the components of $F$ are convex, we also establish a worst-case complexity bound, which reduces to $\mathcal{O}(n\epsilon^{-1})$ for Minimax problems. Numerical results are provided to illustrate the relative efficiency of TRFD in comparison with existing derivative-free solvers for composite nonsmooth optimization.
\end{abstract}

\section{Introduction}\label{sec:1}

\subsection{Problem and Contributions}

We are interested in composite optimization problems of the form
\begin{equation}
\text{Minimize}\,\,f(x)\equiv h(F(x))\,\,\,\text{subject to}\,\,\, x\in\Omega,
\label{eq:first}
\end{equation}
where $F:\mathbb{R}^n \to \mathbb{R}^m$ is assumed to be continuously differentiable with Lipschitz continuous Jacobian, $h : \mathbb{R}^m \to \mathbb{R}$ is a Lipschitz convex function (possibly nonsmooth), and $\Omega \subset \mathbb{R}^n$ is a nonempty closed convex set. Specifically, we assume that $F(\,\cdot\,)$ is only accessible through an exact zeroth-order oracle, meaning that the analytical form of $F(\,\cdot\,)$ is unknown, and that for any $x$, all we can compute is the exact function value $F(x)$. This situation occurs when $F(x)$ is obtained as the output of a black-box software or as the result of some simulation. Examples include aerodynamic shape optimization \cite{karbasian}, optimization of cardiovascular geometries \cite{marsden, russ} or tuning of algorithmic parameters \cite{audet}, just to mention a few. Standard first-order methods for solving composite optimization problems (see, e.g., \cite{Fletcher,Yuan,Cartis}) require the computation of the Jacobian matrices of $F(\,\cdot\,)$ at the iterates, which are not readily available when $F(\,\cdot\,)$ is accessible via a zeroth-order oracle. Therefore, in this setting one needs to rely on derivative-free methods \cite{CSV2,AH,LMW}.

One of the main classes of derivative-free methods is the class of model-based trust-region methods (see, e.g., \cite{CSV}). At each iteration, this type of method builds linear or quadratic interpolation models for the components of $F(\,\cdot\,)$, considering carefully chosen points around the current iterate. Then, the corresponding model of the objective $f(\,\cdot\,)$ is approximately minimized subject to a trust-region constraint. If the resulting trial point provides a sufficient decrease in the objective function, the point is accepted as the next iterate and the radius of the trust-region may increase. Otherwise, the trial point is rejected, and the process is repeated with a reduced trust-region radius. For the class of unconstrained composite optimization problems with $h$ being convex and Lipschitz continuous, Grapiglia, Yuan and Yuan \cite{grapiglia2016derivative} proposed a model-based version of the trust-region method of Fletcher \cite{Fletcher}. They proved that their derivative-free method takes at most $\mathcal{O}\left(n^{2}|\log(\epsilon^{-1})|\epsilon^{-2}\right)$ evaluations of $F(\,\cdot\,)$ to find an $\epsilon$-approximate stationary point of $f(\,\cdot\,)$. Considering a wider class of model-based trust-region methods, Garmanjani, Júdice and Vicente \cite{garmanjani2016trust} established an improved evaluation complexity bound of $\mathcal{O}\left(n^{2}\epsilon^{-2}\right)$, also assuming that $h(\,\cdot\,)$ is convex and Lipschitz continuous. For the case in which $h(\,\cdot\,)$ is not necessarily convex, \textit{Manifold Sampling} trust-region methods have been proposed by Larson and Menickelly \cite{LM1,LM2} under the general assumption that 
\begin{equation*}
h(z)\in\left\{h_{j}(z)\,:\,j\in\left\{1,\ldots,q\right\}\right\},\quad\forall z\in\mathbb{R}^{m},
\end{equation*}
where $h_{j}:\mathbb{R}^{m}\to\mathbb{R}$ is a Lipschitz continuous differentiable function, with $\nabla h_{j}(\,\cdot\,)$ also Lipschitz continuous. In particular, the variants of Manifold Sampling recently proposed in \cite{LM2} are currently the state-of-the-art solvers for derivative-free composite optimization problems. 

In recent years, improved evaluation complexity bounds have been obtained for derivative-free methods designed to smooth unconstrained optimization\footnote{Problem (\ref{eq:first}) with $\Omega=\mathbb{R}^{n}$, $m=1$ and $h(z)=z$.}. Specifically, Grapiglia \cite{grapiglia2,grapiglia} proved evaluation complexity bounds of $\mathcal{O}\left(n\epsilon^{-2}\right)$ for qua\-dratic re\-gu\-larization methods with finite-difference gradient approximations. Moreover, for convex problems, a bound of $\mathcal{O}\left(n\epsilon^{-1}\right)$ was also established in \cite{grapiglia}. Motivated by these advances, in the present work we propose a derivative-free trust-region method based on finite differences for composite problems of the form (\ref{eq:first}). At its $k$-$th$ iteration, our new method (called TRFD) approximates the Jacobian matrix of $F$ at $x_{k}$, $J_{F}(x_{k})$, with a matrix $A_{k}$ obtained via forward finite differences. The stepsize $\tau_{k}$ used in the finite differences and the trust-region radius $\Delta_{k}$ are jointly updated in a way that ensures an error bound
\begin{equation*}
    \|J_{F}(x_{k})-A_{k}\|_{2}\leq\mathcal{O}\left(\Delta_{k}\right),\quad\forall k.
\end{equation*}
Assuming that $f(\,\cdot\,)$ is bounded below by $f_{low}$, and denoting by $L_{h,p}$ the Lipschitz constant of $h(\,\cdot\,)$ with respect to a $p$-norm, and by $L_{J}$ the Lipschitz constant of $J_{F}(\,\cdot\,)$ with respect to the Euclidean norm, we show that TRFD requires no more than
\begin{equation}
\mathcal{O}\left(n c_{2,p}(n)^{2}c_{p,2}(m) L_{h,p} L_{J} (f(x_{0})-f_{low})\epsilon^{-2}\right)
\label{eq:bound1}
\end{equation}
evaluations of $F(\,\cdot\,)$ to find an $\epsilon$-approximate stationary point of $f(\,\cdot\,)$ on $\Omega$, where $c_{2,p}(n)$ and $c_{p,2}(m)$ are positive constants such that
\begin{equation*}
    \|x\|_{2}\leq c_{2,p}(n)\|x\|_{p},\quad\forall x\in\mathbb{R}^{n},\quad\text{and}\quad \|z\|_{p}\leq c_{p,2}(m)\|z\|_{2},\quad\forall z\in\mathbb{R}^{m}.
\end{equation*}
For L1 and Minimax problems, which are composite problems defined respectively by $h(z)=\|z\|_{1}$ and $h(z)=\min_{i=1,\ldots,m}\left\{z_{i}\right\}$, we show that the complexity bound (\ref{eq:bound1}) reduces to $\mathcal{O}\left(n\epsilon^{-2}\right)$ for specific instances of TRFD. This represents an improvement with respect to the complexity bound of $\mathcal{O}\left(n^{2}\epsilon^{-2}\right)$ proved in \cite{garmanjani2016trust} in the context of composite nonsmooth optimization. For the case where $h(\,\cdot\,)$ is monotone and the components of $F(\,\cdot\,)$ are convex functions, we also show that TRFD requires no more than
\begin{equation}
\mathcal{O}\left(n c_{2,p}(n)^{2}c_{p,2}(m) L_{h,p} L_{J}\Delta_{*}^{2}\epsilon^{-1}\right)
    \label{eq:bound2}
\end{equation}
evaluations of $F(\,\cdot\,)$ to find an $\epsilon$-approximate minimizer of $f(\,\cdot\,)$ on $\Omega$, where $\Delta_{*}$ is a sufficiently large upper bound on the trust-region radii. For Minimax problems, we show that the bound (\ref{eq:bound2}) reduces to $\mathcal{O}\left(n \epsilon^{-1}\right)$ for specific instances of TRFD. To the best of our knowledge, this is the first time that evaluation complexity bounds with linear dependence on $n$ are obtained for a deterministic DFO method in the context of composite nonsmooth optimization problems of the form (\ref{eq:first}). Finally, we present numerical results that illustrate the relative efficiency of TRFD on L1 and Minimax problems.
\\[0.2cm]
\subsection{Contents}

The paper is organized as follows. In Section \ref{sec:2}, we prove the relevant auxiliary results. In Section \ref{sec:3}, we analyze the evaluation complexity of the new method for nonconvex
and convex problems. Finally, in Section \ref{sec:4}, we report numerical results for L1 and Minimax problems.

\subsection{Notations}

Throughout the paper, given $p\in \mathbb{N}_{\infty}:=\left\{1,2,3,\ldots\right\}\cup\left\{\infty\right\}$, $\|\,\cdot\,\|_p$ denotes the $p$-norm of vectors or matrices (depending on the context); and $\|\,\cdot\,\|_F$ denotes the Frobenius norm. Given $x \in \Omega$, $y\in\mathbb{R}^{n}$ and $r>0$, we consider the sets 
\begin{equation*}
\Omega - \{x\} := \{s\in\mathbb{R}^{n}: x+s \in \Omega\}\quad\text{and}\quad B_{p}[y;r]=\left\{s\in\mathbb{R}^{n}\,:\,\|s-y\|_{p}\leq r\right\}. 
\end{equation*}
In addition, $[A]_j$ denotes the $j$-$th$ column of the matrix $A\in \mathbb{R}^{m\times n}$, while $[Ad]_i$ denotes the $i$-$th$ coordinate of the vector $Ad \in \mathbb{R}^{m}$. 

\section{Assumptions and Auxiliary results}\label{sec:2}

Through the paper, we will consider the following assumptions:
\\
\\
\textbf{A1.} $\Omega \subset \mathbb{R}^n$ is a nonempty closed convex set.
\\
\textbf{A2.} $F: \mathbb{R}^n \to \mathbb{R}^m$ is differentiable and its Jacobian $J_F$ is $L_J$-Lipschitz on $\mathbb{R}^n$ with respect to the Euclidean norm, that is, $$\|J_F(x)-J_F(y)\|_2 \leq L_J\|x-y\|_2, \quad \forall x,y \in \mathbb{R}^n.$$
\\
\textbf{A3.} $h:\mathbb{R}^m \to \mathbb{R}$ is convex and $L_{h,p}$-Lipschitz continuous on $\mathbb{R}^m$ with respect to the $p$-norm, that is, $$|h(z)-h(w)| \leq L_{h,p}\|z-w\|_p, \quad \forall z,w \in \mathbb{R}^m.$$

\begin{remark}
By A2, given $x,y\in\mathbb{R}^{n}$ we have
\begin{equation*}
\|F(y)-F(x)-J_{F}(x)(y-x)\|_{2}\leq\dfrac{L_{J}}{2}\|y-x\|_{2}^{2}.
\end{equation*}
\label{rem:mais1}
\end{remark}

\begin{remark}
Since all norms are equivalent on Euclidean spaces, given \\$p\in\left\{1,2,\ldots,+\infty\right\}$, there exist positive constants $c_{2,p}(n), c_{p,2}(m) \geq 1$ such that
\begin{equation}
\|x\|_{2}\leq c_{2,p}(n)\|x\|_{p},\quad\forall x\in\mathbb{R}^{n},\quad\text{and}\quad\|z\|_{p}\leq c_{p,2}(m)\|z\|_{2},\quad\forall z\in\mathbb{R}^{m}.
\label{eq:norms}
\end{equation}
\label{rem:mais2}
\end{remark}

\noindent The lemma below provides a necessary condition for a solution of \eqref{eq:first}.
\vspace{0.2cm}
\begin{lemma}\label{motiv}
     Suppose that A1-A3 hold. If $x^*$ is a solution of \eqref{eq:first}, then \begin{equation}\label{sol_crit}
        f(x^*) \leq h(F(x^*)+J_F(x^*)s), \quad \forall s \in \Omega - \{x^*\}.
    \end{equation}
\end{lemma}

\begin{proof}
Suppose by contradiction that $f(x^*) > h(F(x^*)+J_F(x^*)\hat{s})$ for some $\hat{s} \in \Omega - \{x^*\}$. Then $\hat{s}\neq 0$ and there exists $\delta>0$ such that
\begin{equation}
h(F(x^{*})+J_{F}(x^{*})\hat{s})<f(x^{*})-\delta
\label{eq:mais1}
\end{equation}
and
\begin{equation}
\delta\leq L_{h,p}c_{p,2}(m)L_{J}\|\hat{s}\|_{2}^{2}.
\label{eq:mais2}
\end{equation}

\noindent Given $\alpha \in [0,1]$, let $\hat{x}(\alpha)=x^*+\alpha\hat{s}$. Then, using assumptions A3 and A2, and Remarks \ref{rem:mais1} and \ref{rem:mais2}, we have
\begin{align*}
    f(\hat{x}(\alpha)) &= h(F(\hat{x}(\alpha)))=\left[h(F(x^{*}+\alpha\hat{s}))-h(F(x^{*})+J_{F}(x^{*})\alpha\hat{s})\right]\\ &\quad+h(F(x^{*})+J_{F}(x^{*})\alpha\hat{s})\\ 
    &\leq |h(F(x^*+\alpha\hat{s})) - h(F(x^*)+J_{F}(x^*)\alpha \hat{s})| + h(F(x^*)+J_F(x^*)\alpha \hat{s}) \\
    &\leq L_{h,p}\|F(x^*+\alpha\hat{s})-F(x^*)-J_{F}(x^*)\alpha \hat{s}\|_{p} \\ &\quad+ h((1-\alpha)F(x^*) +\alpha(F(x^*)+J_F(x^*)\hat{s})) \\
    &\leq L_{h,p} c_{p,2}(m)\|F(x^*+\alpha\hat{s})-F(x^*)-J_{F}(x^*)\alpha \hat{s}\|_{2} \\ &\quad+ (1-\alpha)h(F(x^*)) +\alpha h(F(x^{*})+J_{F}(x^{*})\hat{s}) \\
    &\leq \frac{L_{h,p} c_{p,2}(m)L_J\|\hat{s}\|_{2}^{2}}{2}\alpha^2+(1-\alpha)f(x^*) + \alpha h(F(x^*)+J_F(x^*)\hat{s}).
\end{align*}
Then, by (\ref{eq:mais1}), it follows that
\begin{eqnarray}
f(\hat{x}(\alpha))&<&\frac{L_{h,p} c_{p,2}(m)L_J\|\hat{s}\|_{2}^{2}}{2}\alpha^{2}+(1-\alpha)f(x^{*})+\alpha\left(f(x^{*})-\delta\right)\nonumber\\
&=& \frac{L_{h,p} c_{p,2}(m)L_J\|\hat{s}\|_{2}^{2}}{2}\alpha^{2}+f(x^{*})-\delta\alpha.
\label{eq:mais3}
\end{eqnarray}
Minimizing the right-hand side of the inequality above with respect to $\alpha$, we obtain
\begin{equation*}
    \alpha^* = \frac{\delta}{L_{h,p} c_{p,2}(m)L_J\|\hat{s}\|_{2}^{2}}.
\end{equation*}
It follows from (\ref{eq:mais2}) that $\alpha^{*}\in [0,1]$. Thus, by (\ref{eq:mais3}) we would have 
\begin{equation*}
f(\hat{x}(\alpha^{*}))< f(x^{*})-\dfrac{\delta^{2}}{2L_{h,p}c_{p,2}(m)L_{J}\|\hat{s}\|_{2}^{2}}<f(x^{*}),  
\end{equation*}
which contradicts the assumption that $x^*$ is a solution of \eqref{eq:first}. Therefore, we conclude that \eqref{sol_crit} is true.
\end{proof}

\noindent Lemma \ref{motiv} motivates the following definition of stationarity, which in the unconstrained case corresponds to the definition considered by Yuan \cite{yuan1985superlinear}.
\vspace{0.2cm}
\begin{definition}
    We say that $x^* \in \Omega$ is a stationary point of $f$ on 
$\Omega$ when $x^*$ satisfies condition \eqref{sol_crit}.
\end{definition}
\vspace{0.2cm}
\noindent Given $p\in\mathbb{N}_{\infty}$ and $r>0$, let us define $\psi_{p,r}:\Omega\to\mathbb{R}$ by \begin{equation}\label{psir}
    \psi_{p,r}(x) = \frac{1}{r}\left(h(F(x)) - \min_{s\in \Omega - \{x\} \atop \|s\|_p \leq r} h(F(x)+J_F(x)s)\right).
\end{equation}
\vspace{0.2cm}
\noindent From the definition of $\psi_{p,r}(\,\cdot\,)$, we have the following result.

\begin{lemma}\label{newfirst}
    Suppose that A1-A3 hold and let $\psi_{p,r}(\,\cdot\,)$ be defined by \eqref{psir}. Then, \\ \\ (a) $\psi_{p,r}(x) \geq 0 \quad \forall x \in \Omega$; \\ (b) $\psi_{p,r}(x^{*})=0$ if, and only if, $x^*$ is a stationary point of $f$ on $\Omega$.
\end{lemma}

\begin{proof}
     Given $x \in \Omega$, we have \begin{align*}
        \min_{s\in \Omega - \{x\} \atop \|s\|_p \leq r} h(F(x)+J_F(x)s) &\leq h(F(x)).
    \end{align*}
    Then, by (\ref{psir}), we have $\psi_{p,r}(x) \geq 0$, and so (a) is true. \\ \\
    To prove (b), let us first assume that $x^{*}\in\Omega$ is a stationary point of $f$. By Definition 2.4, \begin{equation*}
        h(F(x^*)) \leq h(F(x^*)+J_F(x^*)s) \quad \forall s \in \Omega - \{x^*\},
    \end{equation*} 
    and so,
        \begin{align*}
        h(F(x^*)) &\leq \min_{s\in \Omega - \{x^*\} \atop \|s\|_p \leq r} h(F(x^*)+J_F(x^*)s).
    \end{align*} 
    Therefore, in view of (\ref{psir}), we have $\psi_{p,r}(x^*)\leq 0$. Combining this fact with (a), we conclude that $\psi_{p,r}(x^*)=0$. \\ \\
    Now, suppose that $\psi_{p,r}(x^*)=0$. Then, if \fbox{$\tilde{s}\in (\Omega - \{x^*\})\cap B_{p}[0;r]$}, we have 
    \begin{align}
        h(F(x^*)) = \min_{s\in \Omega - \{x^*\} \atop \|s\|_p \leq r} h(F(x^*)+J_F(x^*)s) \leq h(F(x^*)+J_F(x^*)\tilde{s}). \label{eq:case1}
    \end{align}
    On the other hand, suppose that \fbox{$\tilde{s} \in (\Omega - \{x^*\}) \setminus B_{p}[0;r]$}, and define $\gamma=r/\|\tilde{s}\|_{p}$. Notice that $\gamma\in (0,1)$. Then, by (\ref{eq:case1}) and A3 we have
    \begin{align*}
        h(F(x^*)) \leq h(F(x^*)+J_F(x^*)\gamma \tilde{s}) &= h((1-\gamma)F(x^*)+\gamma(F(x^*)+J_F(x^*)\tilde{s})) \\
        &\leq (1-\gamma)h(F(x^*)) + \gamma h(F(x^*)+J_F(x^*)\tilde{s}),
    \end{align*} 
    which implies that \begin{equation}\label{eq:case2}
        h(F(x^*)) \leq h(F(x^*)+J_F(x^*)\tilde{s}).
    \end{equation} As a result, combining \eqref{eq:case1} and \eqref{eq:case2}, we conclude that $x^*$ is a stationary point of $f$ on $\Omega$. Therefore, (b) is also true.
\end{proof}

\noindent In view of Lemma \ref{newfirst}, we will use $\psi_{p,r}(x)$ as a stationarity measure for problem \eqref{eq:first}. 
\vspace{0.2cm}
\begin{remark}
When $\Omega=\mathbb{R}^n$, $m=1$ and $h(z)=z \; \forall z \in \mathbb{R}$, then problem \eqref{eq:first} reduces to the smooth unconstrained problem $\min_{x\in\mathbb{R}^{n}}\,F(x)$, for which $\psi_{2,r}(x)=\|\nabla F(x)\|_{2}$.
\end{remark}
\vspace{0.2cm}
In the context of derivative-free optimization, $J_F(x)$ is unknown. Consequently, $\psi_{p,r}(x)$ is not computable. For a given $x\in \Omega$, our new method will compute a matrix $A \approx J_F(x)$ leading to the following approximate stationarity measure
\begin{align}
    \eta_{p,r}(x;A) := \frac{1}{r}\left(h(F(x)) - \min_{s\in \Omega - \{x\} \atop \|s\|_p \leq r} h(F(x)+As)\right).
    \label{eq:criticality}
\end{align}
The error $|\psi_{p,r}(x)-\eta_{p,r}(x;A)|$ depends on how well $A$ approximates $J_F(x)$. The next lemma provides an error bound when $A$ is computed by forward finite differences.
\vspace{0.2cm}
\begin{lemma}\label{firstlem}
    Suppose that A2 holds. Given $x \in \mathbb{R}^n$ and $\tau >0$, let $A \in \mathbb{R}^{m\times n}$ be defined by \begin{equation}\label{fg}
        [A]_j=\frac{F(x+\tau e_j)-F(x)}{\tau}, \quad j=1, ...,n.
    \end{equation} Then \begin{equation}
        \|A - J_F(x)\|_2 \leq \frac{L_J\sqrt{n}}{2}{\tau}.
        \label{fdap}
    \end{equation}
\end{lemma}

\begin{proof} Given $j\in \{1,...,n\}$, it follows from A2 that $$\left\|F(x+\tau e_j) - F(x) - J_F(x)\tau e_j \right\|_2 \leq \frac{L_J}{2}\|\tau e_j\|_2^2 = \frac{L_J}{2}\tau^2.$$ Thus, by \eqref{fg} we have $$\|[A]_j - [J_F(x)]_j\|_2 \leq \frac{L_J}{2}\tau.$$
Consequently,
\begin{equation*}\|A-J_F(x)\|_F^2 = \sum_{j=1}^{n}\|[A]_j-[J_F(x)]_j\|_2^2 \leq n\left(\frac{L_J}{2}\tau\right)^2.\end{equation*} Therefore $$\|A-J_F(x)\|_2 \leq \|A-J_F(x)\|_F \leq \frac{L_J\sqrt{n}}{2}\tau,$$ and so \eqref{fdap} is true.
\end{proof}
\vspace{0.2cm}
\begin{remark}
Definition (\ref{fg}) implies that the $i$-$th$ row of the corresponding matrix $A$ is a forward finite-difference approximation to $\nabla F_{i}(x)$.
\end{remark}
\vspace{0.2cm}
\noindent Using Lemma \ref{firstlem}, we can establish an upper bound for the error $|\psi_{p,r}(x)-\eta_{p,r}(x;A)|$ when $A$ is constructed by forward finite differences.
\vspace{0.2cm}
\begin{lemma}\label{psieta}
    Suppose that A1-A3 hold. Given $x \in \Omega$ and $\tau>0$, let $A\in \mathbb{R}^{m\times n}$ be defined by \eqref{fg}. Then, \begin{equation}\label{eq13}
        \left|\psi_{p,r}(x) - \eta_{p,r}(x;A)\right| \leq \frac{L_{h,p}L_Jc_{p,2}(m)c_{2,p}(n)\sqrt{n}}{2} {\tau}, \quad \forall x \in \Omega.
    \end{equation}
\end{lemma}

\begin{proof}
By A3, it follows that $s\mapsto h(F(x)+J_{F}(x)s)$ defines a continuous function. Then, by A1 and the Weierstrass Theorem, there exists $\Tilde{s} \in \left(\Omega - \{x\}\right)\cap B_{p}[0;r]$ such that $$\min_{s\in \Omega - \{x\} \atop \|s\|_p \leq r} h(F(x)+J_F(x)s) = h(F(x)+J_F(x)\Tilde{s}).$$ Then, by A3 and \eqref{fdap}, 
\begin{align}
        \psi_{p,r}(x) - \eta_{p,r}(x;A) &= \frac{1}{r}\left[h(F(x))-h(F(x)+J_F(x)\Tilde{s}) - \left(h(F(x))-\min_{s\in \Omega - \{x\} \atop \|s\|_p \leq r} h(F(x)+As)\right)\right] \notag\\
        &= \frac{1}{r}\left[\min_{s\in \Omega - \{x\} \atop \|s\|_p \leq r} h(F(x)+As)-h(F(x)+J_F(x)\Tilde{s})\right] \notag\\
        &\leq \frac{1}{r}\left[h(F(x)+A\Tilde{s})-h(F(x)+J_F(x)\Tilde{s})\right] \notag\\
        &\leq \frac{L_{h,p}}{r}\|(A-J_F(x))\Tilde{s}\|_p \notag\\
        &\leq \frac{L_{h,p}c_{p,2}(m)}{r}\|(A-J_F(x))\Tilde{s}\|_2,\notag \\
        &\leq \frac{L_{h,p}c_{p,2}(m)}{r}\|A-J_F(x)\|_2\|\Tilde{s}\|_2,\notag \\
        &\leq \frac{L_{h,p}c_{p,2}(m)}{r}\|A-J_F(x)\|_2c_{2,p}(n)\|\Tilde{s}\|_p,\notag \\
        &\leq \frac{L_{h,p}L_Jc_{p,2}(m)c_{2,p}(n)\sqrt{n}}{2}\tau \label{lemma2.5}.
    \end{align}
\normalsize
    In a similar way, considering $\hat{s} \in (\Omega - \{x\})\cap B_{p}[0;r]$ such that $$\min_{s\in \Omega - \{x\} \atop \|s\|_p \leq r} h(F(x)+As) = h(F(x)+A\hat{s}),$$ we can show that \begin{equation}\label{lemma2.5other}
        \eta_{p,r}(x;A) - \psi_{p,r}(x) \leq \frac{L_{h,p}L_Jc_{p,2}(m)c_{2,p}(n)\sqrt{n}}{2}\tau.
    \end{equation} Combining \eqref{lemma2.5} and \eqref{lemma2.5other}, we see that (\ref{eq13}) is true.
\end{proof}

\noindent Now, using Lemma \ref{psieta}, we can bound $\eta_{p,r}(x;A)$ from below when $\psi_{p,r}(x)>\epsilon$ and $A$ is defined by (\ref{fg}) with $\tau$ being sufficiently small.
\vspace{0.2cm}
\begin{lemma}\label{lemeta}
    Suppose that A1-A3 hold. Given $x\in \Omega$ and $\tau>0$, let $A$ be defined by \eqref{fg}. Given $r,\epsilon>0$, if $\psi_{p,r}(x)>\epsilon$ and
    \begin{equation}\label{tau_suff}
        \tau \leq \frac{\max\left\{2\eta_{p,r}(x;A), \epsilon\right\}}{L_{h,p}L_Jc_{p,2}(m)c_{2,p}(n)\sqrt{n}},
    \end{equation}
    then $$\eta_{p,r}(x;A) > \frac{\epsilon}{2}.$$
\end{lemma}
\begin{proof}
    From Lemma \ref{psieta} and \eqref{tau_suff}, we have
\begin{align}
    \psi_{p,r}(x) \leq \left|\psi_{p,r}(x)-\eta_{p,r}(x;A)\right| + \eta_{p,r}(x;A) \leq \max\left\{\eta_{p,r}(x;A), \frac{\epsilon}{2}\right\} + \eta_{p,r}(x;A).
    \label{eq:eta_eps}
\end{align}
Suppose that $\eta_{p,r}(x;A) \leq \frac{\epsilon}{2}$. Then from \eqref{eq:eta_eps} we would have $\psi_{p,r}(x) \leq \epsilon$, contradicting the assumption that $\psi_{p,r}(x) > \epsilon$. Therefore, we must have $\eta_{p,r}(x;A) >\epsilon/2$. 
\end{proof}

\noindent Given $0<r_1\leq r_2$, the next lemma establishes the relation between $\eta_{p,r_1}(x;A)$ and $\eta_{p,r_2}(x;A)$ for any given $x\in \Omega$ and $A\in \mathbb{R}^{m\times n}$.
\vspace{0.2cm}
\begin{lemma}
\label{lem:last}
Suppose that A1 and A3 hold. Given $x\in \Omega$, $A\in \mathbb{R}^{m\times n}$ and $0<r_1\leq r_2$, we have $$\eta_{p,r_1}(x;A)\geq \eta_{p,r_2}(x;A).$$
\end{lemma}
\begin{proof} 
Denote \fbox{$\alpha=\frac{r_1}{r_2}$} and let $s^{*}\in \left(\Omega-\left\{x\right\}\right)\cap B_{p}[0;r_2]$ be such that
\begin{equation*}
\min_{s\in \Omega - \{x\} \atop \|s\|_p \leq r_2} h(F(x)+As)=h(F(x)+As^{*}).
\end{equation*}
Then, $\|\alpha s^*\| \leq r_1$ and so $$\min_{s\in \Omega - \{x\} \atop \|s\|_p \leq r_1} h(F(x)+As) \leq h(F(x)+\alpha As^*),$$ which implies that \small \begin{equation}\label{psialpha}
    \eta_{p,r_1}(x;A) = \frac{1}{r_1}\left(h(F(x)) - \min_{s\in \Omega - \{x\} \atop \|s\|_p \leq r_1} h(F(x)+As)\right) \geq \frac{1}{r_1}\left(h(F(x)) - h(F(x) + \alpha As^*)\right).
\end{equation}
\normalsize
On the other hand, using the convexity of $h$ (from A3) and the fact that $\alpha \in (0,1]$, we also have \begin{align}
    h(F(x)+\alpha As^*) &= h((1-\alpha)F(x)+\alpha(F(x)+As^*)) \notag\\
    &\leq (1-\alpha)h(F(x)) + \alpha h(F(x)+As^*) \notag\\
    &= h(F(x)) + \alpha [h(F(x)+As^*) - h(F(x))]\label{conv_h}.
\end{align}
Finally, combining \eqref{psialpha}, \eqref{conv_h} and the definition of $\alpha$, we obtain \small $$ \eta_{p,r_1}(x;A) \geq \frac{\alpha}{r_1} [h(F(x))-h(F(x)+As^*)] = \frac{1}{r_2} [h(F(x))-h(F(x)+As^*)] = \eta_{p,r_2}(x;A),$$ \normalsize and the proof is complete.
\end{proof}

\section{Trust-region method based on finite differences}\label{sec:3}

Leveraging the auxil\-iary results established in the previous section, we propose a derivative-free \textbf{T}rust-\textbf{R}egion Method based on \textbf{F}inite-\textbf{D}ifference Jacobian approximations, which we refer to as TRFD.
\begin{mdframed}
\noindent\textbf{Algorithm 1:} \textbf{TRFD} 
\\[0.2cm]
\noindent\textbf{Step 0.} Given $x_0 \in \Omega$, $\epsilon > 0$, $\sigma > 0$, $\alpha \in (0, 1)$, $\theta\in (0,1]$ and the Lipschitz constant $L_{h,p}$ of $h(\,\cdot\,)$, define
\begin{equation*}
\tau_{0}=\dfrac{\epsilon}{L_{h,p}\sigma c_{p,2}(m)c_{2,p}(n)\sqrt{n}},
\end{equation*}
where $c_{p,2}(m)$ and $c_{2,p}(n)$ satisfy (\ref{eq:norms}). Choose $\Delta_{0}$ and $\Delta_{*}$ such that $\tau_{0}\sqrt{n}\leq\Delta_{0}\leq\Delta_{*}$ and set $k:=0$.
\\[0.2cm]
\noindent\textbf{Step 1.} Construct $A_{k} \in \mathbb{R}^{m\times n}$ with \begin{equation*}
    [A_{k}]_j = \frac{F(x_k+\tau_{k}e_j)-F(x_k)}{\tau_{k}}, \quad j=1,...,n.
\end{equation*}
and compute $\eta_{p,\Delta_{*}}(x_{k};A_{k})$ defined in (\ref{eq:criticality}).
\\[0.2cm]
\noindent\textbf{Step 2.} If $\eta_{p,\Delta_{*}}(x_{k};A_{k})\geq\epsilon/2$, go to Step 3. Otherwise, define $x_{k+1}=x_{k}$, $\Delta_{k+1}=\Delta_{k}$, $\tau_{k+1}=\frac{1}{2}\tau_{k}$, set $k:=k+1$ and go to Step 1.
\\[0.2cm]
\noindent\textbf{Step 3} Let $d_{k}^{*}$ be a solution of the trust-region subproblem 
\begin{equation}
\begin{aligned}
\min_{d\in\mathbb{R}^{n}} \quad & h(F(x_k)+A_{k}d)\\
\textrm{s.t.} \quad & \|d\|_p \leq \Delta_{k}\\
& x_k+d \in \Omega.
\end{aligned}
\label{eq:subproblem}
\end{equation}
Compute $d_{k}\in B_{p}[0;\Delta_{k}]\cap\left(\Omega-\left\{x_{k}\right\}\right)$ such that
\begin{equation}\label{sdcdf}
    h(F(x_k))-h(F(x_k)+A_{k}d_{k}) \geq \theta\left[h(F(x_{k}))-h(F(x_{k})+A_{k}d_{k}^{*})\right].
    \end{equation}
\noindent\textbf{Step 4.} Compute \begin{equation}\label{ratio_generalset}
    \rho_{k}=\frac{h(F(x_k))-h(F(x_k+d_{k}))}{h(F(x_k))-h(F(x_k)+A_{k}d_{k})}.
\end{equation} 
If $\rho_{k}\geq\alpha$, define $x_{k+1}=x_{k}+d_{k}$, $\Delta_{k+1}=\min\left\{2\Delta_{k},\Delta_{*}\right\}$, $\tau_{k+1}=\tau_{k}$, set $k:=k+1$ and go to Step 1. 
\\[0.2cm]
\noindent\textbf{Step 5} Set $x_{k+1}=x_{k}$, $\Delta_{k+1}=\frac{1}{2}\Delta_{k}$. If $\tau_{k}\sqrt{n}\leq\Delta_{k+1}$, define $\tau_{k+1}=\tau_{k}$, $A_{k+1}=A_{k}$, set $k:=k+1$ and go to Step 3. Otherwise, define $\tau_{k+1}=\frac{1}{2}\tau_{k}$, set $k:=k+1$ and go to Step 1.
\end{mdframed}
\vspace{0.3cm}
In TRFD, we have four types of iterations:
\vspace{0.2cm}
\begin{enumerate}
\item\textbf{Unsuccessful iterations of type I} ($\mathcal{U}^{(1)}$): those where $\eta_{p,\Delta_{*}}(x_{k};A_{k})<\epsilon/2$.
\item\textbf{Successful iterations} ($\mathcal{S}$): those where $\eta_{p,\Delta_{*}}(x_{k};A_{k})\geq\epsilon/2$ and $\rho_{k}\geq\alpha$.
\item\textbf{Unsuccessful iterations of type II} ($\mathcal{U}^{(2)}$): those where $\eta_{p,\Delta_{*}}(x_{k};A_{k})\geq\epsilon/2$, $\rho_{k}<\alpha$, and $\tau_{k}\sqrt{n}\leq\Delta_{k+1}$.
\item\textbf{Unsuccessful iterations of type III} ($\mathcal{U}^{(3)}$): those where $\eta_{p,\Delta_{*}}(x_{k};A_{k})\geq\epsilon/2$, $\rho_{k}<\alpha$, and $\tau_{k}\sqrt{n}>\Delta_{k+1}$.
\end{enumerate}
\vspace{0.2cm}
\noindent The lemma below shows that the finite-difference stepsize $\tau_{k}$ is always bounded from above by $\Delta_{k}/\sqrt{n}$.
\vspace{0.2cm}
\begin{lemma}
\label{lem:3.1}
Given $T\geq 1$, let $\left\{\tau_{k}\right\}_{k=0}^{T}$ and $\left\{\Delta_{k}\right\}_{k=0}^{T}$ be generated by TRFD.
Then
\begin{equation}
\tau_{k}\sqrt{n}\leq\Delta_{k},
\label{eq:3.1}
\end{equation}
\text{for} $k=0,\ldots,T$.
\end{lemma}
\begin{proof}
We will prove this result by induction over $k$. In view of the choice of $\Delta_{0}$ at Step 0 of TRFD, we see that (\ref{eq:3.1}) is true for $k=0$. Assuming that (\ref{eq:3.1}) is true for some $k\in\left\{0,\ldots,T-1\right\}$, we will show that it is also true for $k+1$. Considering our classification of iterations, we have four possible cases.
\\[0.2cm]
\noindent\textbf{Case I:} $k\in\mathcal{U}^{(1)}$.
\\[0.2cm]
\noindent In this case, by Step 2 of TRFD, we have $\tau_{k+1}=\frac{1}{2}\tau_{k}$ and $\Delta_{k+1}=\Delta_{k}$. Thus, by the induction assumption, 
\begin{equation*}
\tau_{k+1}\sqrt{n}=\frac{1}{2}\tau_{k}\sqrt{n}<\tau_{k}\sqrt{n}\leq\Delta_{k}=\Delta_{k+1},
\end{equation*}
that is, (\ref{eq:3.1}) holds for $k+1$.
\\[0.2cm]
\noindent\textbf{Case II:} $k\in\mathcal{S}$.
\\[0.2cm]
In this case, by Step 4 of TRFD, we have $\tau_{k+1}=\tau_{k}$ and $\Delta_{k+1}\geq\Delta_{k}$. Thus, by the induction assumption,
\begin{equation*}
\tau_{k+1}\sqrt{n}=\tau_{k}\sqrt{n}\leq\Delta_{k}\leq\Delta_{k+1},
\end{equation*}
which means that (\ref{eq:3.1}) is true for $k+1$.
\\[0.2cm]
\noindent\textbf{Case III:} $k\in\mathcal{U}^{(2)}$
\\[0.2cm]
In this case, by Step 5 of TRFD, we have $\tau_{k}\sqrt{n}\leq\Delta_{k+1}$, and $\tau_{k+1}=\tau_{k}$. Thus
\begin{equation*}
\tau_{k+1}\sqrt{n}=\tau_{k}\sqrt{n}\leq\Delta_{k+1},
\end{equation*}
that is, (\ref{eq:3.1}) is true for $k+1$.
\\[0.2cm]
\noindent\textbf{Case IV:} $k\in\mathcal{U}^{(3)}$.
\\[0.2cm]
In this case, by Step 5 of TRFD we have $\tau_{k+1}=\frac{1}{2}\tau_{k}$ and $\Delta_{k+1}=\frac{1}{2}\Delta_{k}$. Thus, by the induction assumption, 
\begin{equation*}
\tau_{k+1}\sqrt{n}=\frac{1}{2}\tau_{k}\sqrt{n}\leq\frac{1}{2}\Delta_{k}=\Delta_{k+1},
\end{equation*}
and so (\ref{eq:3.1}) is true for $k+1$, which concludes the proof. 
\end{proof}
\vspace{0.2cm}
\noindent In view of Lemmas \ref{firstlem} and \ref{lem:3.1}, the matrices $A_{k}$ in TRFD satisfy
\begin{equation*}
    \|J_{F}(x_{k})-A_{k}\|_{2}\leq\dfrac{L_{J}}{2}\Delta_{k}\quad\forall k.
\end{equation*}
Thanks to this error bound, we can derive the following sufficient condition for an iteration to be successful.
\vspace{0.2cm}
\begin{lemma}
\label{lem:3.2}
Suppose that A1-A3 hold. If $\psi_{p,\Delta_{*}}(x_{k})>\epsilon$ and 
\begin{equation}
\Delta_{k}\leq\dfrac{(1-\alpha)\theta\eta_{p,\Delta_{*}}(x_{k};A_{k})}{L_{h,p}L_{J}c_{p,2}(m)c_{2,p}(n)^{2}},
\label{eq:3.2}
\end{equation}
then $k\in\mathcal{S}$.
\end{lemma}

\begin{proof}
From Step 0 of TRFD, we have $\alpha\in (0,1)$ and $\theta\in (0,1]$. Then, it follows from Lemma \ref{lem:3.1}, (\ref{eq:3.2}) and $c_{2,p}(n)\geq 1$ that
\begin{equation*}
\tau_{k}\leq\dfrac{\Delta_{k}}{\sqrt{n}}\leq\dfrac{2\eta_{p,\Delta_{*}}(x_{k};A_{k})}{L_{h,p}L_{J}c_{p,2}(m)c_{2,p}(n)\sqrt{n}}.
\end{equation*}
Since $\psi_{p,\Delta_{*}}(x_{k})>\epsilon$, by Lemma \ref{lemeta} we get
\begin{equation*}
\eta_{p,\Delta_{*}}(x_{k};A_{k})\geq\epsilon/2.
\end{equation*}
Therefore, to conclude that $k\in S$, it remains to show that $\rho_{k}\geq\alpha$. On the one hand, by A3, (\ref{eq:norms}) and A2 we have
\begin{eqnarray*}
& &h(F(x_{k}+d_{k}))-h(F(x_{k})+A_{k}d_{k})\nonumber\\
&=& h(F(x_{k}+d_{k}))-h(F(x_{k})+J_{F}(x_{k})d_{k})+h(F(x_{k})+J_{F}(x_{k})d_{k})-h(F(x_{k})+A_{k}d_{k})\nonumber\\
&\leq & \left|h(F(x_{k}+d_{k}))-h(F(x_{k})+J_{F}(x_{k})d_{k})\right|+\left|h(F(x_{k})+J_{F}(x_{k})d_{k})-h(F(x_{k})+A_{k}d_{k})\right|\nonumber\\
&\leq & L_{h,p}\|F(x_{k}+d_{k})-F(x_{k})-J_{F}(x_{k})d_{k}\|_{p}+L_{h,p}\|(J_{F}(x_{k})-A_{k})d_{k}\|_{p}\nonumber\\
&\leq & L_{h,p}c_{p,2}(m)\|F(x_{k}+d_{k})-F(x_{k})-J_{F}(x_{k})d_{k}\|_{2}+L_{h,p}c_{p,2}(m)\|J_{F}(x_{k})-A_{k}\|_{2}\|d_{k}\|_{2}\nonumber\\
&\leq & L_{h,p}c_{p,2}(m)\frac{L_{J}}{2}\|d_{k}\|_{2}^{2}+L_{h,p}c_{p,2}(m)\frac{L_{J}\sqrt{n}}{2}\tau_{k}\|d_{k}\|_{2}\nonumber\\
&\leq & \left(0.5\right)L_{h,p}L_{J}c_{p,2}(m)c_{2,p}(n)^{2}\|d_{k}\|_{p}^{2}+\left(0.5\right)L_{h,p}L_{J}c_{p,2}(m)c_{2,p}(n)\tau_{k}\sqrt{n}\|d_{k}\|_{p}\nonumber\\
&\leq & \left(0.5\right)L_{h,p}L_{J}c_{p,2}(m)c_{2,p}(n)^{2}\Delta_{k}^{2}+\left(0.5\right)L_{h,p}L_{J}c_{p,2}(m)c_{2,p}(n)\tau_{k}\sqrt{n}\Delta_{k}.
\end{eqnarray*}
\normalsize
Then, by Lemma \ref{lem:3.1} and $c_{2,p}(n)\geq 1$, we have
\begin{eqnarray}
h(F(x_{k}+d_{k}))-h(F(x_{k})+A_{k}d_{k})&\leq &\left(0.5\right)L_{h,p}L_{J}c_{p,2}(m)c_{2,p}(n)\left(c_{2,p}(n)+1\right)\Delta_{k}^{2}\nonumber\\
&\leq &L_{h,p}L_{J}c_{p,2}(m)c_{2,p}(n)^{2}\Delta_{k}^{2}.
\label{eq:3.4}
\end{eqnarray}
On the other hand, by (\ref{sdcdf}) and (\ref{eq:criticality}) we have
\begin{eqnarray*}
h(F(x_{k}))-h(F(x_{k})+A_{k}d_{k})&\geq &\theta\left[h(F(x_{k}))-h(F(x_{k})+A_{k}d_{k}^{*})\right]\\ 
&=&\theta\Delta_{k}\left[\dfrac{1}{\Delta_{k}}\left(h(F(x_{k}))-h(F(x_{k})+A_{k}d_{k}^{*})\right)\right]\\
&=&\theta\Delta_{k}\eta_{p,\Delta_{k}}(x_{k};A_{k}).
\end{eqnarray*}
Since $\Delta_{k}\leq\Delta_{*}$, it follows from Lemma \ref{lem:last} that
\begin{equation}
h(F(x_{k}))-h(F(x_{k})+A_{k}d_{k})\geq\theta\Delta_{k}\eta_{p,\Delta_{*}}(x_{k};A_{k}).
    \label{eq:3.5}
\end{equation}
Now, combining (\ref{ratio_generalset}), (\ref{eq:3.4}), (\ref{eq:3.5}) and (\ref{eq:3.2}), we obtain
\begin{eqnarray*}
1-\rho_{k}&=&\dfrac{h(F(x_{k}))-h(F(x_{k})+A_{k}d_{k})-\left[h(F(x_{k}))-h(F(x_{k}+d_{k}))\right]}{h(F(x_{k}))-h(F(x_{k})+A_{k}d_{k})}\\
&=&\dfrac{h(F(x_{k}+d_{k}))-h(F(x_{k})+A_{k}d_{k})}{h(F(x_{k}))-h(F(x_{k})+A_{k}d_{k})}\leq \dfrac{L_{h,p}L_{J}c_{p,2}(m)c_{2,p}(n)^{2}\Delta_{k}^{2}}{\theta\Delta_{k}\eta_{p,\Delta_{*}}(x_{k};A_{k})}\\
& = &\dfrac{L_{h,p}L_{J}c_{p,2}(m)c_{2,p}(n)^{2}\Delta_{k}}{\theta\eta_{p,\Delta_{*}}(x_{k};A_{k})}\leq 1-\alpha.
\end{eqnarray*}
Therefore, $\rho_{k}\geq\alpha$, and we conclude that $k\in\mathcal{S}$.
\end{proof}
\vspace{0.2cm}
\noindent Now we can obtain a lower bound on the trust-region radii.
\vspace{0.2cm}
\begin{lemma}
\label{lem:3.3}
Suppose that A1-A3 hold and, given $T\geq 1$, let $\left\{\Delta_{k}\right\}_{k=0}^{T}$ be generated by TRFD. If 
\begin{equation*}
\psi_{p,\Delta_{*}}(x_{k})>\epsilon,\quad\text{for}\quad k=0,\ldots,T-1,
\label{eq:3.6}
\end{equation*}
then
\begin{equation}
\Delta_{k}\geq\dfrac{(1-\alpha)\theta\epsilon}{4L_{h,p}\max\left\{\sigma,L_{J}\right\}c_{p,2}(m)c_{2,p}(n)^{2}}\equiv\Delta_{\min}(\epsilon),\quad\text{for}\quad k=0,\ldots,T.
\label{eq:3.7}
\end{equation}
\end{lemma}

\begin{proof}
First, let us prove by induction that 
\begin{equation}
\Delta_{k}\geq\min\left\{\Delta_{0},\dfrac{(1-\alpha)\theta\epsilon}{4L_{h,p}L_{J}c_{p,2}(m)c_{2,p}(n)^{2}}\right\}\equiv\tilde{\Delta}_{\min}(\epsilon),\quad\text{for}\quad k=0,\ldots,T.
\label{eq:3.8}
\end{equation}
Clearly, the inequality in (\ref{eq:3.8}) is true for $k=0$. Suppose that the inequality in (\ref{eq:3.8}) is true for some $k\in\left\{0,\ldots,T-1\right\}$. If $k\in\mathcal{U}^{(1)}$, then $\Delta_{k+1}=\Delta_{k}\geq\tilde{\Delta}_{\min}(\epsilon)$ and so (\ref{eq:3.8}) holds for $k+1$. Now, suppose that $k\notin\mathcal{U}^{(1)}$. In this case, we have $\eta_{p,\Delta_{*}}(x_{k};A_{k})\geq\epsilon/2$. Thus, if 
\begin{equation}
\Delta_{k}\leq\dfrac{(1-\alpha)\theta\epsilon}{2L_{h,p}L_{J}c_{p,2}(m)c_{2,p}(n)^{2}},
\label{eq:3.9}
\end{equation}
then by Lemma \ref{lem:3.2}, $k\in\mathcal{S}$. Consequently, Step 4 of TRFD and the induction assumption imply that
\begin{equation*}
\Delta_{k+1}=\min\left\{2\Delta_{k},\Delta_{*}\right\}\geq\min\left\{\Delta_{k},\Delta_{*}\right\}=\Delta_{k}\geq\tilde{\Delta}_{\min}(\epsilon).
\end{equation*}
Now, suppose that (\ref{eq:3.9}) is not true. Since in any case we have $\Delta_{k+1}\geq\frac{1}{2}\Delta_{k}$, we will have
\begin{equation*}
\Delta_{k+1}\geq\dfrac{1}{2}\Delta_{k}>\dfrac{(1-\alpha)\theta\epsilon}{4L_{h,p}L_{J}c_{p,2}(m)c_{2,p}(n)^{2}}\geq\tilde{\Delta}_{\min}(\epsilon).
\end{equation*}
This shows that (\ref{eq:3.8}) is true. Finally, it follows from Step 0 of TRFD that
\begin{equation*}
\Delta_{0}\geq  \tau_{0}\sqrt{n}=\dfrac{\epsilon}{L_{h,p}\sigma c_{p,2}(m)c_{2,p}(n)}\\
\geq \dfrac{(1-\alpha)\theta\epsilon}{4L_{h,p}\max\left\{\sigma,L_{J}\right\}c_{p,2}(m)c_{2,p}(n)^{2}}.
\end{equation*}
Thus, from the definition of $\tilde{\Delta}_{\min}(\epsilon)$ in (\ref{eq:3.8}), we see that
\begin{equation}
\tilde{\Delta}_{\min}(\epsilon)\geq\dfrac{(1-\alpha)\theta\epsilon}{4L_{h,p}\max\left\{\sigma,L_{J}\right\}c_{p,2}(m)c_{2,p}(n)^{2}}=\Delta_{\min}(\epsilon).
\label{eq:3.10}
\end{equation}
Then, combining (\ref{eq:3.8}) and (\ref{eq:3.10}), we conclude that (\ref{eq:3.7}) is true.
\end{proof}

\subsection{Worst-Case Complexity Bound for Nonconvex Problems}

\noindent Given $j\in\left\{0,1,2,\ldots,\right\}$, let
\begin{eqnarray*}
\mathcal{S}_{j}&=&\left\{0,1,\ldots,j\right\}\cap\mathcal{S},\\
\mathcal{U}^{(i)}_{j}&=&\left\{0,1,\ldots,j\right\}\cap\mathcal{U}^{(i)},\quad i\in\left\{1,2,3\right\}.
\end{eqnarray*}
Also, let 
\begin{equation}
    T_{g}(\epsilon)=\inf\left\{k\in\mathbb{N}\,:\,\psi_{p,\Delta_{*}}(x_{k})\leq\epsilon\right\}
    \label{eq:hitting}
\end{equation}
be the index of the first iteration in which $\left\{x_{k}\right\}_{k\geq 0}$ reaches an $\epsilon$-approximate stationary point, if it exists. Our goal is to obtain a finite upper bound for $T_{g}(\epsilon)$. Assuming that $T_{g}(\epsilon)\geq 1$, it follows from the notation above that
\begin{eqnarray}
T_{g}(\epsilon)&=& \left|\mathcal{S}_{T_{g}(\epsilon)-1}\cup\left(\mathcal{U}_{T_{g}(\epsilon)-1}^{(1)}\cup \mathcal{U}_{T_{g}(\epsilon)-1}^{(2)}\cup \mathcal{U}_{T_{g}(\epsilon)-1}^{(3)}\right)\right|\nonumber\\
&\leq & \left|\mathcal{S}_{T_{g}(\epsilon)-1}\right|+\left|\mathcal{U}_{T_{g}(\epsilon)-1}^{(1)}\cup\mathcal{U}_{T_{g}(\epsilon)-1}^{(3)}\right|+\left|\mathcal{U}_{T_{g}(\epsilon)-1}^{(2)}\cup \mathcal{U}_{T_{g}(\epsilon)-1}^{(3)}\right|.
\label{eq:motivation}
\end{eqnarray}
In the next three lemmas, we will provide upper bounds for each of the three terms in (\ref{eq:motivation}). To that end, let us consider the following additional assumption:
\\[0.2cm]
\textbf{A4.} There exists $f_{low}\in\mathbb{R}$ such that $f(x)\geq f_{low}$ for all $x\in\mathbb{R}^{n}$.
\\[0.2cm]
The next lemma provides an upper bound on $\left|\mathcal{S}_{T_{g}(\epsilon)-1}\right|$.
\vspace{0.2cm}
\begin{lemma}
\label{lem:3.4}
Suppose that A1-A4 hold and that $T_g(\epsilon)\geq 1$. Then
\begin{equation}
|\mathcal{S}_{T_{g}(\epsilon)-1}|\leq \dfrac{8L_{h,p}\max\left\{\sigma,L_{J}\right\}c_{p,2}(m)c_{2,p}(n)^{2}(f(x_{0})-f_{low})}{\alpha (1-\alpha)\theta^{2}}\epsilon^{-2}.
\label{eq:3.11}
\end{equation}
\end{lemma}

\begin{proof}
Let $k\in S_{T_{g}(\epsilon)-1}$, that is, $\eta_{p,\Delta_{*}}(x_{k};A_{k})\geq\epsilon/2$ and $\rho_{k}\geq\alpha$. Then, by (\ref{sdcdf}), (\ref{eq:criticality}), $\Delta_{k}\leq\Delta_{*}$ and Lemma \ref{lem:last}, we have
\begin{eqnarray*}
f(x_{k})-f(x_{k+1})&=&h(F(x_{k}))-h(F(x_{k}+d_{k}))\\
&\geq &\alpha\left[h(F(x_{k}))-h(F(x_{k})+A_{k}d_{k})\right]\\
&\geq &\alpha\theta\left[h(F(x_{k}))-h(F(x_{k})+A_{k}d_{k}^{*})\right]\\
&=&\alpha\theta\Delta_{k}\left[\frac{1}{\Delta_{k}}\left(h(F(x_{k}))-h(F(x_{k})+A_{k}d_{k}^{*})\right)\right]\\
&=&\alpha\theta\Delta_{k}\eta_{p,\Delta_{k}}(x_{k};A_{k})\\
&\geq &\alpha\theta\Delta_{k}\eta_{p,\Delta_{*}}(x_{k};A_{k})\\
&\geq &\dfrac{\alpha\theta\epsilon}{2}\Delta_{k}.
\end{eqnarray*}
Consequently, it follows from Lemma \ref{lem:3.3} that
\begin{equation}
f(x_{k})-f(x_{k+1})\geq \dfrac{\alpha(1-\alpha)\theta^{2}}{8L_{h,p}\max\left\{\sigma,L_{J}\right\}c_{p,2}(m)c_{2,p}(n)^{2}}\epsilon^{2}\quad\text{when}\,\,k\in\mathcal{S}_{T_{g}(\epsilon)-1}.
\label{eq:3.12}
\end{equation}
Let $\mathcal{S}_{T_{g}(\epsilon)-1}^{c}=\left\{0,1,\ldots,T_{g}(\epsilon)-1\right\}\setminus \mathcal{S}_{T_{g}(\epsilon)-1}$. Notice that, when $k\in\mathcal{S}_{T_{g}(\epsilon)-1}^{c}$, then $f(x_{k+1})=f(x_{k})$. Thus, it follows from A4 and (\ref{eq:3.12}) that
\begin{eqnarray*}
f(x_{0})-f_{low}&\geq & f(x_{0})-f(x_{T_{g}(\epsilon)})=\sum_{k=0}^{T_{g}(\epsilon)-1}f(x_{k})-f(x_{k+1})\\
&= & \sum_{k\in\mathcal{S}_{T_{g}(\epsilon)-1}}f(x_{k})-f(x_{k+1})+\sum_{k\in\mathcal{S}_{T_{g}(\epsilon)-1}^{c}}f(x_{k})-f(x_{k+1})\\
&=&\sum_{k\in\mathcal{S}_{T_{g}(\epsilon)-1}}f(x_{k})-f(x_{k+1})\\
&\geq &|\mathcal{S}_{T_{g}(\epsilon)-1}|\dfrac{\alpha(1-\alpha)\theta^{2}}{8L_{h,p}\max\left\{\sigma,L_{J}\right\}c_{p,2}(m)c_{2,p}(n)^{2}}\epsilon^{2},
\end{eqnarray*}
which implies that (\ref{eq:3.11}) is true.
\end{proof}
\vspace{0.2cm}
The next lemma provides an upper bound on  $\left|\mathcal{U}_{T_{g}(\epsilon)-1}^{(1)}\cup\mathcal{U}_{T_{g}(\epsilon)-1}^{(3)}\right|$.
\begin{lemma}
\label{lem:3.6}
Suppose that A1-A3 hold and that $T_{g}(\epsilon)\geq 2$. If $T\in\left\{2,\ldots,T_{g}(\epsilon)\right\}$, then
\begin{equation}
\left|\mathcal{U}_{T-1}^{(1)}\cup\mathcal{U}_{T-1}^{(3)}\right|\leq\left\lceil\left|\log_{2}\left(\frac{\tau_{0}\sqrt{n}}{\Delta_{\min}(\epsilon)}\right)\right|\right\rceil,
\label{eq:3.18}
\end{equation}
where $\Delta_{\min}(\epsilon)$ is defined in \eqref{eq:3.7}.
\end{lemma}

\begin{proof}
Suppose by contradiction that 
\begin{equation}
\left|\mathcal{U}_{T-1}^{(1)}\cup\mathcal{U}_{T-1}^{(3)}\right|>\left\lceil\left|\log_{2}\left(\frac{\tau_{0}\sqrt{n}}{\Delta_{\min}(\epsilon)}\right)\right|\right\rceil.
\label{eq:3.19}
\end{equation}
Notice that 
\begin{equation}
\left|\mathcal{U}_{0}^{(1)}\cup\mathcal{U}_{0}^{(3)}\right|\leq 1\quad\text{and}\quad
\left|\mathcal{U}_{k+1}^{(1)}\cup\mathcal{U}_{k+1}^{(3)}\right|\leq \left|\mathcal{U}_{k}^{(1)}\cup\mathcal{U}_{k}^{(3)}\right|+1, \quad \forall k.
\label{eq:3.20}
\end{equation}
It follows from (\ref{eq:3.19}) and (\ref{eq:3.20}) that there exists $k_{*}\in\left\{0,\ldots,T-2\right\}$ such that 
\begin{equation*}
\left|\mathcal{U}_{k_{*}}^{(1)}\cup\mathcal{U}_{k_{*}}^{(3)}\right|=\left\lceil\left|\log_{2}\left(\frac{\tau_{0}\sqrt{n}}{\Delta_{\min}(\epsilon)}\right)\right|\right\rceil.
\end{equation*}
In view of (\ref{eq:3.20}), \fbox{for any $k\in\mathbb{N}\cap [k_{*},T-1]$} we have
\begin{equation*}
\left|\mathcal{U}_{k}^{(1)}\cup\mathcal{U}_{k}^{(3)}\right|\geq\left|\mathcal{U}_{k_{*}}^{(1)}\cup\mathcal{U}_{k_{*}}^{(3)}\right|\geq\left|\log_{2}\left(\frac{\tau_{0}\sqrt{n}}{\Delta_{\min}(\epsilon)}\right)\right| \geq -\log_{2}\left(\frac{\Delta_{\min}(\epsilon)}{\tau_{0}\sqrt{n}}\right).
\end{equation*}
Thus
\begin{equation*}
-\left|\mathcal{U}_{k}^{(1)}\cup\mathcal{U}_{k}^{(3)}\right|\leq \log_{2}\left(\frac{\Delta_{\min}(\epsilon)}{\tau_{0}\sqrt{n}}\right),
\end{equation*}
and so 
\begin{equation}
\tau_{k}=(0.5)^{\left|\mathcal{U}_{k}^{(1)}\cup\,\mathcal{U}_{k}^{(3)}\right|}\tau_{0}=2^{-\left|\mathcal{U}_{k}^{(1)}\cup\,\mathcal{U}_{k}^{(3)}\right|}\tau_{0}\leq\dfrac{\Delta_{\min}(\epsilon)}{\sqrt{n}}.
\label{eq:3.21}
\end{equation}
By (\ref{eq:3.21}), the definition of $\Delta_{\min}(\epsilon)$ in (\ref{eq:3.7}), and $c_{2,p}(n)\geq 1$, we have
\begin{equation*}
\tau_{k}\leq \dfrac{\epsilon}{L_{h,p}L_{J}c_{p,2}(m)c_{2,p}(n)\sqrt{n}}.    
\end{equation*}
Since we also have $\psi_{p,\Delta_{*}}(x_{k})>\epsilon$, it follows from Lemma \ref{lemeta} that $\eta_{p,\Delta_{*}}(x_{k})\geq \epsilon/2$. Therefore, the $k$-$th$ iteration is not an unsuccessful iteration of type I, i.e., \fbox{$k\notin\mathcal{U}^{(1)}$}. In addition, (\ref{eq:3.21}) and Lemma \ref{lem:3.3} imply that
\begin{equation*}
\tau_{k}\sqrt{n}\leq\Delta_{\min}(\epsilon)\leq\Delta_{k+1},
\end{equation*}
which means that the $k$-$th$ iteration is not an unsuccessful iteration of type III, i.e., \fbox{$k\notin\mathcal{U}^{(3)}$}. In summary, \fbox{$k\notin\mathcal{U}^{(1)}\cup\mathcal{U}^{(3)}$} and so
\begin{equation*}
\left|\mathcal{U}_{k}^{(1)}\cup\,\mathcal{U}_{k}^{(3)}\right|=\left|\mathcal{U}_{k-1}^{(1)}\cup\,\mathcal{U}_{k-1}^{(3)}\right|.
\end{equation*}
Thus, for any $k_{*}<k\leq T-1$, 
\begin{equation*}
\left|\mathcal{U}_{k}^{(1)}\cup\,\mathcal{U}_{k}^{(3)}\right|=\left|\mathcal{U}_{k-1}^{(1)}\cup\,\mathcal{U}_{k-1}^{(3)}\right|=\ldots=\left|\mathcal{U}_{k_{*}}^{(1)}\cup\,\mathcal{U}_{k_{*}}^{(3)}\right|
\end{equation*}
In particular,
\begin{equation*}
\left|\mathcal{U}_{T-1}^{(1)}\cup\,\mathcal{U}_{T-1}^{(3)}\right|=\left|\mathcal{U}_{k_{*}}^{(1)}\cup\,\mathcal{U}_{k_{*}}^{(3)}\right|=\left\lceil\left|\log_{2}\left(\frac{\tau_{0}\sqrt{n}}{\Delta_{\min}(\epsilon)}\right)\right|\right\rceil,
\end{equation*}
contradicting (\ref{eq:3.19}).
\end{proof}

\begin{remark}
By the definition of $\tau_{0}$ (at Step 0 of TRFD) and the definition of $\Delta_{\min}(\epsilon)$ in \eqref{eq:3.7}, we have
\begin{equation}
\dfrac{\tau_{0}\sqrt{n}}{\Delta_{\min}(\epsilon)}
=\dfrac{4\max\left\{\sigma,L_{J}\right\}c_{2,p}(n)}{\sigma(1-\alpha)\theta}.
\label{eq:3.22}
\end{equation}
\normalsize
\end{remark}
\vspace{0.2cm}
The lemma below provides an upper bound on  $\left|\mathcal{U}_{T_{g}(\epsilon)-1}^{(2)}\cup\,\mathcal{U}_{T_{g}(\epsilon)-1}^{(3)}\right|$.

\begin{lemma}
\label{lem:3.5}
Suppose that A1-A3 hold and that $T_{g}(\epsilon)\geq 1$. If $T\in\left\{1,\ldots,T_{g}(\epsilon)\right\}$, then
\begin{equation}
\left|\mathcal{U}_{T-1}^{(2)}\cup\,\mathcal{U}_{T-1}^{(3)}\right|\leq\log_{2}\left(\dfrac{4L_{h,p}\max\left\{\sigma,L_{J}\right\}c_{p,2}(m)c_{2,p}(n)^{2}\Delta_{0}}{(1-\alpha)\theta}\epsilon^{-1}\right)+|\mathcal{S}_{T-1}|.
\label{eq:3.13}
\end{equation}
\end{lemma}

\begin{proof}
By the update rules for $\Delta_{k}$ in TRFD, we have
\begin{eqnarray*}
\Delta_{k+1}&=&\frac{1}{2}\Delta_{k},\,\,\text{if}\,\,k\in\mathcal{U}_{T-1}^{(2)}\cup\mathcal{U}_{T-1}^{(3)},\\
\Delta_{k+1}&=&\Delta_{k},\quad\text{if}\,\,k\in\mathcal{U}_{T-1}^{(1)},\\
\Delta_{k+1}&\leq & 2\Delta_{k},\,\,\,\text{if}\,\,k\in\mathcal{S}_{T-1}.
\end{eqnarray*}
In addition, by Lemma \ref{lem:3.3} we have
\begin{equation*}
\Delta_{k}\geq \Delta_{\min}(\epsilon)\quad\text{for}\quad k=0,\ldots,T,
\end{equation*}
where $\Delta_{\min}(\epsilon)$ is defined in (\ref{eq:3.7}). Thus, considering $\nu_{k}=1/\Delta_{k}$, it follows that
\begin{eqnarray}
2\nu_{k}&=&\nu_{k+1},\quad\text{if}\,\,k\in\mathcal{U}_{T-1}^{(2)}\cup\mathcal{U}_{T-1}^{(3)},\label{eq:3.14}\\
\nu_{k}&=&\nu_{k+1},\quad\text{if}\,\,k\in\mathcal{U}_{T-1}^{(1)},\label{eq:3.15}\\
\frac{1}{2}\nu_{k}&\leq & \nu_{k+1},\quad\text{if}\,\,k\in\mathcal{S}_{T-1},
\label{eq:3.16}
\end{eqnarray}
and
\begin{equation}
\nu_{k}\leq\Delta_{\min}(\epsilon)^{-1}\quad\text{for}\quad k=0,\ldots,T.
\label{eq:3.17}
\end{equation}
In view of (\ref{eq:3.14})-(\ref{eq:3.17}), we have
\begin{equation*}
2^{\left|\mathcal{U}_{T-1}^{(2)}\cup\,\,\mathcal{U}_{T-1}^{(3)}\right|}\left(0.5\right)^{\left|\mathcal{S}_{T-1}\right|}\nu_{0}\leq\nu_{T}\leq\Delta_{\min}(\epsilon)^{-1}.
\end{equation*}
Then, taking the logarithm in both sides we get
\begin{equation*}
\left|\mathcal{U}_{T-1}^{(2)}\cup\,\mathcal{U}_{T-1}^{(3)}\right|-|\mathcal{S}_{T-1}|\leq\log_{2}\left(\frac{\Delta_{\min}(\epsilon)^{-1}}{\nu_{0}}\right)=\log_{2}\left(\frac{\Delta_{0}}{\Delta_{\min}(\epsilon)}\right),
\end{equation*}
which together with (\ref{eq:3.7}) implies that (\ref{eq:3.13}) is true.
\end{proof}
\vspace{0.2cm}
\noindent Now, combining the previous results, we obtain the following worst-case complexity bound on the number of iterations required by TRFD to find an $\epsilon$-approximate stationary point.
\vspace{0.2cm}
\begin{theorem}
\label{thm:3.1}
Suppose that A1-A4 hold and let $T_{g}(\epsilon)$ be defined by \eqref{eq:hitting}. Then
\begin{eqnarray}
T_{g}(\epsilon)&\leq &\dfrac{16L_{h,p}\max\left\{\sigma,L_{J}\right\}c_{p,2}(m)c_{2,p}(n)^{2}(f(x_{0})-f_{low})}{\alpha (1-\alpha)\theta^{2}}\epsilon^{-2}+\left\lceil\left|\log_{2}\left(\dfrac{4\max\left\{\sigma,L_{J}\right\}c_{2,p}(n)}{\sigma(1-\alpha)\theta}\right)\right|\right\rceil\nonumber\\
& &+\log_{2}\left(\dfrac{4L_{h,p}\max\left\{\sigma,L_{J}\right\}c_{p,2}(m)c_{2,p}(n)^{2}\Delta_{0}}{(1-\alpha)\theta}\epsilon^{-1}\right)+1.
\label{eq:3.23}
\end{eqnarray}
\normalsize
\end{theorem}

\begin{proof}
If $T_{g}(\epsilon)\leq 1$, then (\ref{eq:3.23}) is clearly true. Let us assume that $T_{g}(\epsilon)\geq 2$. By (\ref{eq:motivation}),
\begin{equation*}
T_{g}(\epsilon)\leq  \left|\mathcal{S}_{T_{g}(\epsilon)-1}\right|+\left|\mathcal{U}_{T_{g}(\epsilon)-1}^{(1)}\cup\mathcal{U}_{T_{g}(\epsilon)-1}^{(3)}\right|+\left|\mathcal{U}_{T_{g}(\epsilon)-1}^{(2)}\cup \mathcal{U}_{T_{g}(\epsilon)-1}^{(3)}\right|.
\end{equation*}
Then, (\ref{eq:3.23}) follows directly from Lemmas \ref{lem:3.4}, \ref{lem:3.6} and \ref{lem:3.5}, together with (\ref{eq:3.22}).
\end{proof}
\noindent Since each iteration of TRFD requires at most $(n+1)$ evaluations of $F(\,\cdot\,)$, from Theorem \ref{thm:3.1} we obtain the following upper bound on the total number of evaluations of $F(\,\cdot\,)$ required by TRFD to find an $\epsilon$-approximate stationary point.
\vspace{0.2cm}
\begin{corollary}
Suppose that A1-A4 hold and let $FE_{T_{g}(\epsilon)-1}$ be the total number of function evaluations executed by TRFD up to the $(T_{g}(\epsilon)-1)$-st iteration. Then
\small
\begin{eqnarray*}
FE_{T_{g}(\epsilon)-1}&\leq &(n+1)\left[\dfrac{16L_{h,p}\max\left\{\sigma,L_{J}\right\}c_{p,2}(m)c_{2,p}(n)^{2}(f(x_{0})-f_{low})}{\alpha (1-\alpha)\theta^{2}}\epsilon^{-2}\right.
\\ & &\left.+\left\lceil\left|\log_{2}\left(\dfrac{4\max\left\{\sigma,L_{J}\right\}c_{2,p}(n)}{\sigma(1-\alpha)\theta}\right)\right|\right\rceil
+\log_{2}\left(\dfrac{4L_{h,p}\max\left\{\sigma,L_{J}\right\}c_{p,2}(m)c_{2,p}(n)^{2}\Delta_{0}}{(1-\alpha)\theta}\epsilon^{-1}\right)+1\right].
\end{eqnarray*}
\label{cor:3.1}    
\end{corollary}
\vspace{0.2cm}
\noindent In view of Corollary \ref{cor:3.1}, TRFD needs no more than 
\begin{equation*}
\mathcal{O}\left(n\,c_{2,p}(n)^{2}c_{p,2}(m)L_{h,p}L_{J}(f(x_{0})-f_{low})\epsilon^{-2}\right)
\label{eq:3.24}
\end{equation*}
evaluations of $F(\,\cdot\,)$ to find $x_{k}$ such that $\psi_{p,\Delta_{*}}(x_{k})\leq\epsilon$.
In what follows, Tables \ref{tab:1} and \ref{tab:2} specify this complexity bound for the L1 and Minimax problems with respect to different choices of the $p$-norm used in TRFD.

\begin{table}[h!]
\centering
\begin{tabular}{|c|c|c|c|c|}
\hline
\multicolumn{5}{|c|}{\textbf{L1 Problems}: case $h(z)=\|z\|_{1}$ $\forall z\in\mathbb{R}^{m}$} \\ \hline
$p$-norm in TRFD & $L_{h,p}$ & $c_{p,2}(m)$ & $c_{2,p}(n)$ & Evaluation Complexity Bound \\ \hline
$p=1$ & 1 & $\sqrt{m}$ & 1 & $\mathcal{O}\left(n\sqrt{m}L_{J}(f(x_{0})-f_{low})\epsilon^{-2}\right)$ \\ \hline
$p=2$ & $\sqrt{m}$ & 1 & 1 & $\mathcal{O}\left(n\sqrt{m}L_{J}(f(x_{0})-f_{low})\epsilon^{-2}\right)$ \\ \hline
$p=\infty$ & $m$ & 1 & $\sqrt{n}$ & $\mathcal{O}\left(n^{2}mL_{J}(f(x_{0})-f_{low})\epsilon^{-2}\right)$ \\ \hline
\end{tabular}
\caption{Complexity bounds for problems with objective function of the form  $f(\,\cdot\,)=\|F(\,\cdot\,)\|_{1}$.}
\label{tab:1}
\end{table}

\begin{table}[h!]
\centering
\begin{tabular}{|c|c|c|c|c|}
\hline
\multicolumn{5}{|c|}{\textbf{Minimax Problems}: case $h(z)=\max_{i=1,\ldots,m}\left\{z_{i}\right\}$ $\forall z\in\mathbb{R}^{m}$} \\ \hline
$p$-norm in TRFD & $L_{h,p}$ & $c_{p,2}(m)$ & $c_{2,p}(n)$ & Evaluation Complexity Bound \\ \hline
$p=1$ & 1 & $\sqrt{m}$ & 1 & $\mathcal{O}\left(n\sqrt{m}L_{J}(f(x_{0})-f_{low})\epsilon^{-2}\right)$ \\ \hline
$p=2$ & 1 & 1 & 1 & $\mathcal{O}\left(n L_{J}(f(x_{0})-f_{low})\epsilon^{-2}\right)$ \\ \hline
$p=\infty$ & 1 & 1 & $\sqrt{n}$ & $\mathcal{O}\left(n^{2} L_{J}(f(x_{0})-f_{low})\epsilon^{-2}\right)$ \\ \hline
\end{tabular}
\caption{Complexity bounds for problems with objective function of the form $f(\,\cdot\,)=\max_{i=1,\ldots,m}\left\{F_{i}(\,\cdot\,)\right\}$.}
\label{tab:2}
\end{table}
Notice that in both cases, considering TRFD with $p=1$ or $p=2$, we obtain evaluation complexity bounds of $\mathcal{O}\left(n\epsilon^{-2}\right)$, with linear dependence on the number of variables $n$. This represents an improvement over the bound of $\mathcal{O}\left(n^2\epsilon^{-2}\right)$ established in \cite{garmanjani2016trust} for a model-based derivative-free trust-region method for composite nonsmooth optimization.
\subsection{Worst-Case Complexity Bound for Convex Problems}

Let us consider the additional assumptions:
\\[0.2cm]
\textbf{A5.} $F_{i}(\,\cdot\,)$ is convex for $i=1,\ldots,m$.
\\[0.2cm]
\textbf{A6.} $h(\,\cdot\,)$ is monotone, i.e., $h(u)\leq h(v)$ if $u_{i}\leq v_{i}$ for $i=1,\ldots,m$.
\\[0.2cm]
\textbf{A7.} $f(\,\cdot\,)=h(F(\,\cdot\,))$ has a global minimizer $x^{*}$ on $\Omega$ and 
\begin{equation*}
    D_{0}\equiv\sup_{x\in\mathcal{L}_{f}(x_{0})}\left\{\|x-x^{*}\|_{p}\right\}<+\infty,
\end{equation*}
for $\mathcal{L}_{f}(x_{0})=\left\{x\in\Omega\,:\,f(x)\leq f(x_{0})\right\}$.
\vspace{0.2cm}
\\
\noindent The lemma below establishes the relationship between the stationarity measure and the functional residual when the reference radius $r$ is sufficiently large.
\vspace{0.2cm}
\begin{lemma}
    \label{lem:3.2.1}
    Suppose that A1, A2, A5, A6 and A7 hold, and let $x_k\in\mathcal{L}_{f}(x_{0})$. If $r\geq D_{0}$, then
    \begin{equation*}
        \psi_{p,r}(x_k)\geq\dfrac{1}{r}(f(x_k)-f(x^{*})).
    \end{equation*}
\end{lemma}

\begin{proof}
    By A2 and A5, for each $i\in\left\{1,\ldots,m\right\}$ we have
    \begin{equation*}
        F_{i}(x_k+s)\geq F(x_k)+\langle\nabla F_{i}(x_k),s\rangle,\quad\forall s\in\Omega-\left\{x_k\right\}.
    \end{equation*}
    Thus, it follows from A6 that
    \begin{equation*}
        h(F(x_k+s))\geq h(F(x_k)+J_{F}(x_k)s),\quad\forall s\in\Omega-\{x_k\},
    \end{equation*}
    and so
    \begin{equation}
        \min_{s\in \Omega - \{x_k\} \atop \|s\|_p \leq r}\,h(F(x_k+s))\geq\min_{s\in \Omega - \{x_k\} \atop \|s\|_p \leq r}\,h(F(x_k)+J_{F}(x_k)s).
        \label{eq:3.2.1}
    \end{equation}
    Let $s^{*}=x^{*}-x_k$. Then $s^{*}\in\Omega-\left\{x_k\right\}$ and, by A7 and $r\geq D_{0}$, we also have $\|s^{*}\|_{p}\leq D_{0}\leq r$. Therefore
    \begin{equation}
        \min_{s\in \Omega - \{x_k\} \atop \|s\|_p \leq r}\,h(F(x_k+s))=h(F(x_k+s^{*}))=h(F(x^{*}))=f(x^{*}).
        \label{eq:3.2.2}
    \end{equation}
    Combining (\ref{eq:3.2.1}) and (\ref{eq:3.2.2}), we obtain
    \begin{equation*}
    f(x^{*})\geq\min_{s\in \Omega - \{x_k\} \atop \|s\|_p \leq r}\,h(F(x_k)+J_{F}(x_k)s),
    \end{equation*}
    which implies that
    \begin{equation*}
        \psi_{p,r}(x_k)=\frac{1}{r}\left(h(F(x_k))-\min_{s\in \Omega - \{x_k\} \atop \|s\|_p \leq r}\,h(F(x_k)+J_{F}(x_k)s)\right)\geq\frac{1}{r}\left(f(x_k)-f(x^{*})\right),
    \end{equation*}
    which concludes the proof.
\end{proof}
\noindent The lemma below provides a lower bound on the approximate stationarity measure in terms of the functional residual.
\vspace{0.2cm}
\begin{lemma}
    \label{lem:3.2.2}
    Suppose that A1-A3 and A5-A7 hold, and let $\left\{x_{k}\right\}$ be generated by TRFD. If $\Delta_{*}\geq D_{0}$, then
    \begin{equation*}
        \eta_{p,\Delta_{*}}(x_{k};A_{k})\geq\dfrac{f(x_{k})-f(x^{*})}{\left(\frac{L_{J}}{\sigma}+1\right)\Delta_{*}}
    \end{equation*}
    whenever $k\notin\mathcal{U}^{(1)}$.
\end{lemma}

\begin{proof}
    Suppose that $k\notin\mathcal{U}^{(1)}$. In this case, we have $\eta_{p,\Delta_{*}}(x_{k};A_{k})\geq\epsilon/2$. Then, it follows from Lemma \ref{psieta} and from the definition of $\tau_0$ at Step 0 of TRFD that 
    \begin{eqnarray*}
        \psi_{p,\Delta_{*}}(x_{k})&\leq &|\psi_{p,\Delta_{*}}(x_{k})-\eta_{p,\Delta_{*}}(x_{k};A_{k})|+|\eta_{p,\Delta_{*}}(x_{k};A_{k})|\\
        &\leq &\dfrac{L_{h,p}L_{J}c_{p,2}(m)c_{2,p}(n)\sqrt{n}}{2}\tau_{k}+\eta_{p,\Delta_{*}}(x_{k};A_{k})\\
        &\leq &\dfrac{L_{h,p}L_{J}c_{p,2}(m)c_{2,p}(n)\sqrt{n}}{2}\tau_{0}+\eta_{p,\Delta_{*}}(x_{k};A_{k})\\
        &=&\dfrac{L_{J}\epsilon}{2\sigma}+\eta_{p,\Delta_{*}}(x_{k};A_{k})\\
        &\leq&\left(\frac{L_{J}}{\sigma}+1\right)\eta_{p,\Delta_{*}}(x_{k};A_{k}).
    \end{eqnarray*}
    Therefore, by Lemma \ref{lem:3.2.1}, we obtain
    \begin{equation*}
        \eta_{p,\Delta_{*}}(x_{k};A_{k})\geq\dfrac{\psi_{p,\Delta_{*}}(x_{k})}{\left(\frac{L_{J}}{\sigma}+1\right)}\geq\dfrac{f(x_{k})-f(x^{*})}{\left(\frac{L_{J}}{\sigma}+1\right)\Delta_{*}}.
    \end{equation*}
    Therefore, the statement is proved.
\end{proof}
\noindent Next we establish an upper bound for $\frac{f(x_{k})-f(x^{*})}{\Delta_{k}}$ when the functional residual is sufficiently large.
\vspace{0.2cm}
\begin{lemma}
    \label{lem:3.2.3}
    Suppose that A1-A3 and A5-A7 hold, and let $\left\{x_{k}\right\}_{k=0}^{T}$ be generated by TRFD. If $\Delta_{*}\geq\Delta_{0}$ and
    \begin{equation}
        f(x_{k})-f(x^{*})>\Delta_{*}\epsilon\quad\text{for}\,\,k=0,\ldots,T-1,
        \label{eq:3.2.3}
    \end{equation}
    then
    \small
    \begin{equation}
    \left(\frac{1}{\Delta_{k}}\right)(f(x_{k})-f(x^{*}))\leq\max\left\{\left(\frac{1}{\Delta_{0}}\right)(f(x_{0})-f(x^{*})),\dfrac{2\left(\frac{L_{J}}{\sigma}+1\right)\Delta_{*}L_{h,p}L_{J}c_{p,2}(m)c_{2,p}(n)^{2}}{(1-\alpha)\theta}\right\}\equiv\beta
    \label{eq:3.2.4}
    \end{equation}
    \normalsize
    for $k=0,\ldots,T$.
\end{lemma}

\begin{proof}
    By the definition of $\beta$, (\ref{eq:3.2.4}) is true for $k=0$. Suppose that (\ref{eq:3.2.4}) is true for some $k\in\left\{0,\ldots,T-1\right\}$. Let us show that it is also true for $k+1$.
    \\[0.3cm]
    \noindent\textbf{Case 1:} $k\in\mathcal{U}^{(1)}\cup\mathcal{S}$
    \\[0.2cm]
    In this case, we have $\Delta_{k+1}\geq\Delta_{k}$. Since $f(x_{k+1})\leq f(x_{k})$, it follows that
    \begin{equation*}
    \left(\frac{1}{\Delta_{k+1}}\right)(f(x_{k+1})-f(x^{*}))\leq\left(\frac{1}{\Delta_{k}}\right)(f(x_{k})-f(x^{*}))\leq\beta,
    \end{equation*}
    where the last inequality is the induction assumption. Therefore, (\ref{eq:3.2.4}) holds for $k+1$ in this case.
    \\[0.2cm]
    \noindent\textbf{Case 2:} $k\in\mathcal{U}^{(2)}\cup\mathcal{U}^{(3)}$
    \\[0.2cm]
    In this case we have 
    \begin{equation}
          \Delta_{k+1}=\frac{1}{2}\Delta_{k}.
        \label{eq:3.2.5}
    \end{equation}
    In addition, in view of (\ref{eq:3.2.3}) and $\Delta_{*}\geq D_{0}$, it follows from Lemma \ref{lem:3.2.1} that $\psi_{p,\Delta_{*}}(x_{k})>\epsilon$. Therefore, we must have
    \begin{equation}
        \Delta_{k}>\dfrac{(1-\alpha)\theta\eta_{p,\Delta_{*}}(x_{k};A_{k})}{L_{h,p}L_{J}c_{p,2}(m)c_{2,p}(n)^{2}}
        \label{eq:3.2.6}
    \end{equation}
    since otherwise, by Lemma \ref{lem:3.2}, we would have $k\in\mathcal{S}$, contradicting our assumption that $k\in\mathcal{U}^{(2)}\cup\mathcal{U}^{(3)}$. Notice that (\ref{eq:3.2.6}) is equivalent to
    \begin{equation*}
     \left(\frac{1}{\Delta_{k}}\right)\eta_{p,\Delta_{*}}(x_{k};A_{k})<\dfrac{L_{h,p}L_{J}c_{p,2}(m)c_{2,p}(n)^{2}}{(1-\alpha)\theta}.   
    \end{equation*}
    Finally, it follows from (\ref{eq:3.2.5}), Lemma \ref{lem:3.2.2} and (\ref{eq:3.2.6}) that
    \begin{eqnarray*}
        \left(\frac{1}{\Delta_{k+1}}\right)(f(x_{k+1})-f(x^{*}))&=&\left(\frac{2}{\Delta_{k}}\right)(f(x_{k+1})-f(x^{*}))\leq \left(\frac{2}{\Delta_{k}}\right)(f(x_{k})-f(x^{*}))\\
        &\leq &\dfrac{2\left(\frac{L_{J}}{\sigma}+1\right)\Delta_{*}}{\Delta_{k}}\eta_{p,\Delta_{*}}(x_{k};A_{k})\\
        & < & 2\left(\frac{L_{J}}{\sigma}+1\right)\Delta_{*}\dfrac{L_{h,p}L_{J}c_{p,2}(m)c_{2,p}(n)^{2}}{(1-\alpha)\theta}\\
        & = &\dfrac{2\left(\frac{L_{J}}{\sigma}+1\right)\Delta_{*}L_{h,p}L_{J}c_{p,2}(m)c_{2,p}(n)^{2}}{(1-\alpha)\theta}\\
        &\leq &\beta,
    \end{eqnarray*}
that is, (\ref{eq:3.2.4}) also holds for $k+1$ in this case.
\end{proof}

Let 
\begin{equation}
T_{f}(\epsilon)=\inf\left\{k\in\mathbb{N}\,:\,f(x_{k})-f(x^{*})\leq\Delta_{*}\epsilon\right\}
\label{eq:3.2.7}
\end{equation}
be the index of the first iteration in which $\left\{x_{k}\right\}_{k\geq 0}$ reaches a $\Delta_{*}\epsilon$-approximate solution of (\ref{eq:first}), if it exists. Our goal is to establish a finite upper bound for $T_{f}(\epsilon)$. In this context, the lemma below provides an upper bound on $\left|\mathcal{S}_{T_{f}(\epsilon)-1}\right|$.
\vspace{0.2cm}
\begin{lemma}
    \label{lem:3.2.4}
    Suppose that A1-A7 hold. Given $\epsilon>0$, if $T_{f}(\epsilon)\geq 2$ and $\Delta_{*}\geq D_{0}$, then
    \begin{equation}
    \left|\mathcal{S}_{T_{f}(\epsilon)-1}\right|\leq 1+\dfrac{\left(\frac{L_{J}}{\sigma}+1\right)\beta}{\alpha\theta}\epsilon^{-1},
    \label{eq:3.2.8}
    \end{equation}
    where $\beta$ is defined in \eqref{eq:3.2.4}.
\end{lemma}

\begin{proof}
    Let $k\in\mathcal{S}_{T_{f}(\epsilon)-2}$. By Lemmas \ref{lem:last}, \ref{lem:3.2.2} and \ref{lem:3.2.3}, we have
    \begin{eqnarray}
    f(x_{k})-f(x_{k+1})&\geq &\alpha\left[h(F(x_{k}))-h(F(x_{k})+A_{k}d_{k})\right]\nonumber\\
    &\geq &\alpha\theta\left[h(F(x_{k}))-h(F(x_{k})+A_{k}d_{k}^{*})\right]\nonumber\\
    &=&\alpha\theta\Delta_{k}\eta_{p,\Delta_{k}}(x_{k};A_{k})\nonumber\\
    &\geq &\alpha\theta\Delta_k\eta_{p,\Delta_{*}}(x_{k};A_{k})\nonumber\\
    &\geq &\alpha\theta\Delta_{k}\dfrac{f(x_{k})-f(x^{*})}{\left(\frac{L_{J}}{\sigma}+1\right)\Delta_{*}}\nonumber\\
    &=&\dfrac{\alpha\theta (f(x_{k})-f(x^{*}))^{2}}{\left(\frac{L_{J}}{\sigma}+1\right)\Delta_{*}\left(\frac{1}{\Delta_{k}}\right)(f(x_{k})-f(x^{*}))}\nonumber\\
    &\geq &\dfrac{\alpha\theta (f(x_{k})-f(x^{*}))^{2}}{\left(\frac{L_{J}}{\sigma}+1\right)\Delta_{*}\beta}.
    \label{eq:3.2.9}
    \end{eqnarray}
    Denoting $\delta_{k}=f(x_{k})-f(x^{*})$, (\ref{eq:3.2.9}) becomes
    \begin{equation*}
    \delta_{k}-\delta_{k+1}\geq\dfrac{\alpha\theta}{\left(\frac{L_{J}}{\sigma}+1\right)\Delta_{*}\beta}\delta_{k}^{2}.
    \end{equation*}
    Consequently,
    \begin{equation}
    \dfrac{1}{\delta_{k+1}}-\dfrac{1}{\delta_{k}}=\dfrac{\delta_{k}-\delta_{k+1}}{\delta_{k}\delta_{k+1}}\geq\dfrac{\frac{\alpha\theta}{\left(\frac{L_{J}}{\sigma}+1\right)\Delta_{*}\beta}\delta_{k}^{2}}{\delta_{k}^{2}}=\dfrac{\alpha\theta}{\left(\frac{L_{J}}{\sigma}+1\right)\Delta_{*}\beta}.
    \label{eq:3.2.10}    
    \end{equation}
    Since $\delta_{k+1}=\delta_{k}$ for any $k\in \{0,...,T_f(\epsilon)-2\}\setminus \mathcal{S}_{T_{f}(\epsilon)-2}$, it follows from (\ref{eq:3.2.10}) that
    \begin{eqnarray*}
        \dfrac{1}{\delta_{T_{f}(\epsilon)-1}}-\dfrac{1}{\delta_{0}}&=&\sum_{k=0}^{T_{f}(\epsilon)-2}\dfrac{1}{\delta_{k+1}}-\dfrac{1}{\delta_{k}}=\sum_{k\in\mathcal{S}_{T_{f}(\epsilon)-2}}\dfrac{1}{\delta_{k+1}}-\dfrac{1}{\delta_{k}}\\
        &\geq &\left|\mathcal{S}_{T_{f}(\epsilon)-2}\right|\dfrac{\alpha\theta}{\left(\frac{L_{J}}{\sigma}+1\right)\Delta_{*}\beta}.
    \end{eqnarray*}
    Therefore
    \begin{equation*}
        \Delta_{*}\epsilon<f(x_{T_{f}(\epsilon)-1})-f(x^{*})=\delta_{T_f(\epsilon)-1}\leq\dfrac{\left(\frac{L_{J}}{\sigma}+1\right)\Delta_{*}\beta}{\alpha\theta\left|\mathcal{S}_{T_{f}(\epsilon)-2}\right|},
    \end{equation*}
    which implies that
    \begin{equation*}
        \left|\mathcal{S}_{T_{f}(\epsilon)-1}\right|\leq 1+\left|\mathcal{S}_{T_{f}(\epsilon)-2}\right|\leq 1+\dfrac{\left(\frac{L_{J}}{\sigma}+1\right)\beta}{\alpha\theta}\epsilon^{-1},
    \end{equation*}
    that is, \eqref{eq:3.2.8} is true.
\end{proof}
\noindent The next lemma establishes the relationship between $T_{f}(\epsilon)$ and $T_{g}(\epsilon)$.
\vspace{0.2cm}
\begin{lemma}\label{lem:tf}
    Suppose that A2, A5, A6 and A7 hold, and let $T_f(\epsilon)$ and $T_{g}(\epsilon)$ be defined by \eqref{eq:3.2.7} and \eqref{eq:hitting}, respectively. If $\Delta_* \geq D_0$, then $T_f(\epsilon) \leq T_{g}(\epsilon)$.
\end{lemma}
\begin{proof}
    Suppose by contradiction that $T_f(\epsilon) > T_{g}(\epsilon).$ In this case, by $\Delta_* \geq D_0$ and Lemma \ref{lem:3.2.1}, we would arrive at the contradiction 
    \begin{equation*}
        \epsilon<\frac{1}{\Delta_{*}}(f(x_{T_{g}(\epsilon)})-f(x^{*})) \leq \psi_{p,\Delta_*}(x_{T_{g}(\epsilon)})\leq\epsilon.
    \end{equation*}
    Therefore, we must have $T_{f}(\epsilon)\leq T_{g}(\epsilon)$.
\end{proof}
\vspace{0.2cm}
\noindent The following theorem gives an upper bound on the number of iterations required by TRFD to reach a $\Delta_{*}\epsilon$-approximate solution of (\ref{eq:first}).
\vspace{0.2cm}
\begin{theorem}
    Suppose that A1-A7 and let $T_f(\epsilon)$ be defined by \eqref{eq:3.2.7}. If $\Delta_* \geq D_0$, then
    \begin{align}
        T_f(\epsilon) \leq & 2\left[1+\frac{(\frac{L_J}{\sigma}+1)\beta}{\alpha\theta}\epsilon^{-1}\right] + \left\lceil\left|\log_{2}\left(\dfrac{4\max\left\{\sigma,L_{J}\right\}c_{2,p}(n)}{\sigma(1-\alpha)\theta}\right)\right|\right\rceil \notag \\ &+ \log_{2}\left(\dfrac{4L_{h,p}\max\left\{\sigma,L_{J}\right\}c_{p,2}(m)c_{2,p}(n)^{2}\Delta_{0}}{(1-\alpha)\theta}\epsilon^{-1}\right),
        \label{eq:3.14.1}
    \end{align}
    where $\beta$ is defined in \eqref{eq:3.2.4}.
    \label{thm:convex}
\end{theorem}
\begin{proof}
    If $T_f(\epsilon) \leq 1$, then \eqref{eq:3.14.1} is true. Let us assume that $T_f(\epsilon) \geq 2$. As in the proof of Theorem \ref{thm:3.1}, we have
    \begin{equation}
        T_f(\epsilon) \leq \left|\mathcal{S}_{T_f(\epsilon)-1}\right|+\left|\mathcal{U}_{T_f(\epsilon)-1}^{(1)}\cup\mathcal{U}_{T_f(\epsilon)-1}^{(3)}\right|+\left|\mathcal{U}_{T_f(\epsilon)-1}^{(2)}\cup \mathcal{U}_{T_f(\epsilon)-1}^{(3)}\right|.\label{eq:3.14.2}
    \end{equation}
    By Lemma \ref{lem:tf}, we have $T_f(\epsilon) \leq T_g(\epsilon)$. Thus, it follows from Lemmas \ref{lem:3.6} and \ref{lem:3.5} that
    \begin{equation}
\left|\mathcal{U}_{T_f(\epsilon)-1}^{(1)}\cup\mathcal{U}_{T_f(\epsilon)-1}^{(3)}\right|\leq\left\lceil\left|\log_{2}\left(\frac{\tau_{0}\sqrt{n}}{\Delta_{\min}(\epsilon)}\right)\right|\right\rceil\label{eq:3.14.4}
\end{equation}
and
\begin{equation}
\left|\mathcal{U}_{T_f(\epsilon)-1}^{(2)}\cup\,\mathcal{U}_{T_f(\epsilon)-1}^{(3)}\right|\leq\log_{2}\left(\dfrac{4L_{h,p}\max\left\{\sigma,L_{J}\right\}c_{p,2}(m)c_{2,p}(n)^{2}\Delta_{0}}{(1-\alpha)\theta}\epsilon^{-1}\right)+|\mathcal{S}_{T_f(\epsilon)-1}|,\label{eq:3.14.3}
    \end{equation}
where $\Delta_{\min}(\epsilon)$ is defined in (\ref{eq:3.7}). Then, combining \eqref{eq:3.14.2}, Lemma \ref{lem:3.2.4}, \eqref{eq:3.14.4}, \eqref{eq:3.14.3} and \eqref{eq:3.22}, we conclude that \eqref{eq:3.14.1} is true.
\end{proof}
\noindent Since each iteration of TRFD requires at most $(n+1)$ evaluations of $F(\,\cdot\,)$, from Theorem \ref{thm:convex} we obtain the following upper bound on the total number of evaluations of $F(\,\cdot\,)$ required by TRFD to find a $\Delta_{*}\epsilon$-approximate solution of (\ref{eq:first}). 
\vspace{0.2cm}
\begin{corollary}\label{coro:3.15}
    Suppose that A1-A7 hold and let $FE_{T_f(\epsilon)-1}$ be the total number of function evaluations executed by TRFD up to the $(T_f(\epsilon)-1)$-$st$ iteration. If $\Delta_* \geq D_0$, then
    \begin{eqnarray*}
FE_{T_f(\epsilon)-1}&\leq &(n+1)\left[2\left[1+\frac{(\frac{L_J}{\sigma}+1)\beta}{\alpha\theta}\epsilon^{-1}\right]+\left\lceil\left|\log_{2}\left(\dfrac{4\max\left\{\sigma,L_{J}\right\}c_{2,p}(n)}{\sigma(1-\alpha)\theta}\right)\right|\right\rceil\right.\\
&    &\left.+\log_{2}\left(\dfrac{4L_{h,p}\max\left\{\sigma,L_{J}\right\}c_{p,2}(m)c_{2,p}(n)^{2}\Delta_{0}}{(1-\alpha)\theta}\epsilon^{-1}\right)\right].
\end{eqnarray*}
\end{corollary}
In view of Corollary \ref{coro:3.15} and the definition of $\beta$ in \eqref{eq:3.2.4}, if $h(\,\cdot\,)$ is monotone and the components $F_i(\,\cdot\,)$ are convex, then TRFD, with a sufficiently large $\Delta_*$, needs no more than $$\mathcal{O}\left(n\,c_{2,p}(n)^{2}c_{p,2}(m)L_{h,p}L_{J}\Delta_*\epsilon^{-1}\right)$$ function evaluations to find $x_k$ such that $$f(x_k)-f(x^*) \leq \Delta_*\epsilon.$$ Thus, given $\epsilon_f > 0$, if we use TRFD with $\epsilon = \epsilon_f/\Delta_*$, then it will need no more than $$\mathcal{O}\left(n\,c_{2,p}(n)^{2}c_{p,2}(m)L_{h,p}L_{J}\Delta_*^2\epsilon_f^{-1}\right)$$ function evaluations to find $x_k$ such that $$f(x_k)-f(x^*) \leq \epsilon_f.$$ Table \ref{tab:3} below specifies the complexity bound for the Minimax problem, which is a composite nonsmooth problem of the form \eqref{eq:first} whose function $h(\,\cdot\,)$ is monotone.
\newpage
\begin{table}[h!]
\centering
\begin{tabular}{|c|c|c|c|c|}
\hline
\multicolumn{5}{|c|}{\textbf{Minimax Problems}: case $h(z)=\max_{i=1,\ldots,m}\left\{z_{i}\right\}$ $\forall z\in\mathbb{R}^{m}$} \\ \hline
$p$-norm in TRFD & $L_{h,p}$ & $c_{p,2}(m)$ & $c_{2,p}(n)$ & Evaluation Complexity Bound \\ \hline
$p=1$ & 1 & $\sqrt{m}$ & 1 & $\mathcal{O}\left(n\sqrt{m} L_{J}\Delta_*^2\epsilon_f^{-1}\right)$ \\ \hline
$p=2$ & 1 & 1 & 1 & $\mathcal{O}\left(n L_{J}\Delta_*^2\epsilon_f^{-1}\right)$ \\ \hline
$p=\infty$ & 1 & 1 & $\sqrt{n}$ & $\mathcal{O}\left(n^{2} L_{J}\Delta_*^2\epsilon_f^{-1}\right)$ \\ \hline
\end{tabular}
\caption{Complexity bounds for problems with objective function of the form $f(\,\cdot\,)=\max_{i=1,\ldots,m}\left\{F_{i}(\,\cdot\,)\right\}$.}
\label{tab:3}
\end{table}
When $\Omega$ is a polyhedron, for $p=1$ and $p=\infty$, the computation of $\eta_{p,\Delta_{*}}(x_{k};\Delta_{k})$ and $d_{k}$ in TRFD can be performed by solving linear programming problems. The complexity bounds in Table \ref{tab:3} suggest that one should use $p=1$ when $\sqrt{m}< n$, and $p=\infty$ otherwise. On the other hand, the best complexity bound, of $\mathcal{O}\left(n\epsilon^{-1}\right)$, is obtained with $p=2$. However, in this case, the computation of $\eta_{p,\Delta_{*}}(x_{k};\Delta_{k})$ and $d_{k}$ requires solving linear problems subject to a quadratic constraint.

\section{Numerical experiments}\label{sec:4}

We performed numerical experiments with Matlab implementations of TRFD. Specifically, two classes of test problems were considered: unconstrained L1 problems (see subsection \ref{sub:l1}) and unconstrained Minimax problems (see subsection \ref{sub:minimax}). We compared TRFD against Manifold Sampling Primal \cite{LM2} and against the derivative-free trust-region method proposed in \cite{grapiglia2016derivative}. For each problem, a budget of 100 simplex gradients was allowed to each solver\footnote{One simplex gradient corresponds to $n+1$ function evaluations, with $n$ being the number of variables of the problem.}. In addition, our implementations of TRFD were 
equipped with the following stopping criteria: $$\Delta_k \leq 10^{-13} \quad \text{or} \quad \eta_{p,\Delta_{*}}(x_k;A_k) \leq 10^{-13}.$$ Implementations are compared using data profiles \cite{more2009benchmarking}\footnote{The data profiles were generated using the code \textit{data\_profile.m}, freely available at the website\\ \url{https://www.mcs.anl.gov/~more/dfo/}.}, where a code $M$ is said to solve a given problem when it reaches $x_M$ such that $$\frac{f(x_0)-f(x_{M})}{f(x_0)-f(x_{Best})} \geq 1-\textit{\text{Tolerance}},$$ where $f(x_{Best})$ is the lowest function value found among all the methods. All experiments were performed with MATLAB (R2023a) on a PC with microprocessor 13-$th$ Gen Intel(R) Core(TM) i5-1345U 1.60 GHz and 32 GB of RAM memory.

\subsection{L1 problems}\label{sub:l1}

Here we considered problems of the form $$\min_{x\in \mathbb{R}^n} \|F(x)\|_1.$$ We tested 53 functions $F:\mathbb{R}^n \to \mathbb{R}^m$ defined by Moré and Wild \cite{more2009benchmarking}, for which $2\leq n \leq 12$ and $2\leq m \leq 65$. The following codes were compared:
\\[0.15cm]
- \textbf{TRFD-L1}: Implementation of TRFD with $p=1$ and parameters $\epsilon=10^{-15}$, $\alpha=0.15$, $L_{h,p}=1$, $\Delta_0=\max\{1, \tau_0\sqrt{n}\}$, $\Delta_{*}=1000$ and $$\sigma=\frac{\epsilon}{L_{h,p}c_{p,2}(m)c_{2,p}(n)\sqrt{n}\sqrt{eps}},$$ where $eps$ is the machine precision, $c_{1,2}(m)=\sqrt{m}$ and $c_{2,1}(n)=1$. The computation of $\eta_{p,\Delta_{*}}(x_{k};\Delta_{k})$ and $d_{k}$ is performed using the MATLAB function \textit{linprog.m}.
\\[0.15cm]
- \textbf{MS-P}: Implementation of Manifold Sampling Primal \cite{LM2}, freely available on \\ GitHub\footnote{\url{https://github.com/POptUS/IBCDFO}.}. The initial parameters are given in the file \textit{check\_inputs\_and\_initialize.m}, while the outer function $h(\,\cdot\,)$ was provided by the file \textit{one\_norm.m}.
\\[0.15cm]
- \textbf{DFL1S}: Implementation of the trust-region method in \cite{grapiglia2016derivative} adapted to the case $h(\,\cdot\,)=\|\,\cdot\,\|_1$.
\\[0.15cm]
Data profiles are shown in Figure \ref{fig:l1_results}. As we can see, in this particular test set, TRFD-L1 outperforms both MS-P and DFL1S, being able to solve a higher percentage of problems within the allowed budget of $100(n+1)$ evaluations of $F(\,\cdot\,)$ across all the tolerances considered.
\begin{figure}[h!]
    \centering
    \subfigure
    {\includegraphics[width=0.33\textwidth]{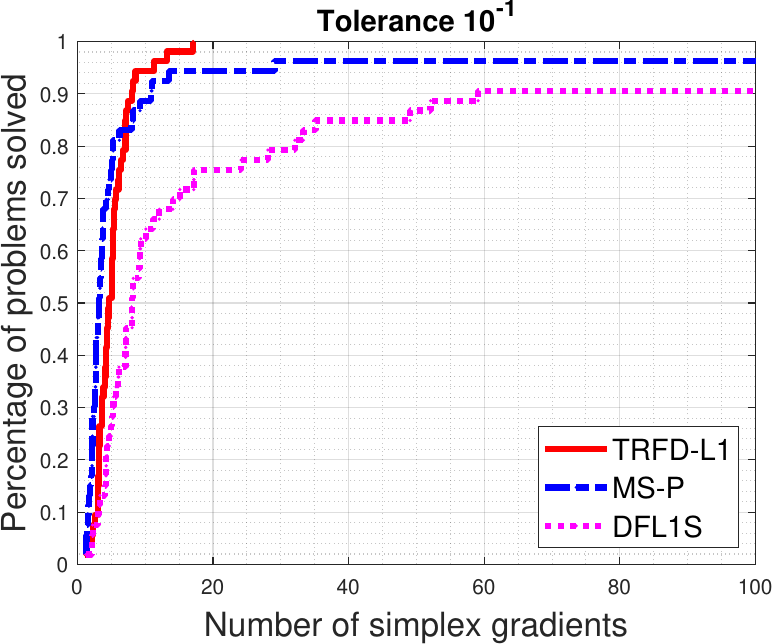}}
    \hspace{5mm}
    \subfigure
    {\includegraphics[width=0.33\textwidth]{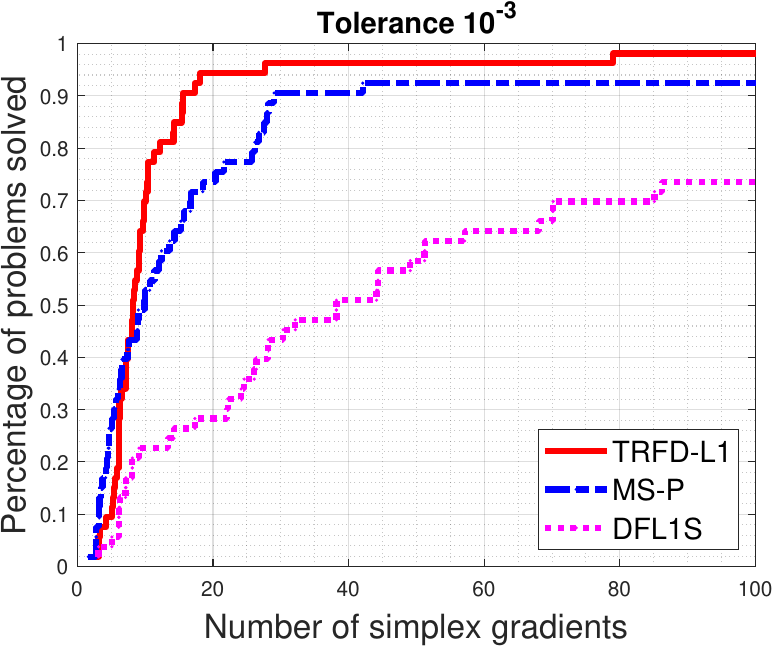}} 
    \subfigure
    {\includegraphics[width=0.33\textwidth]{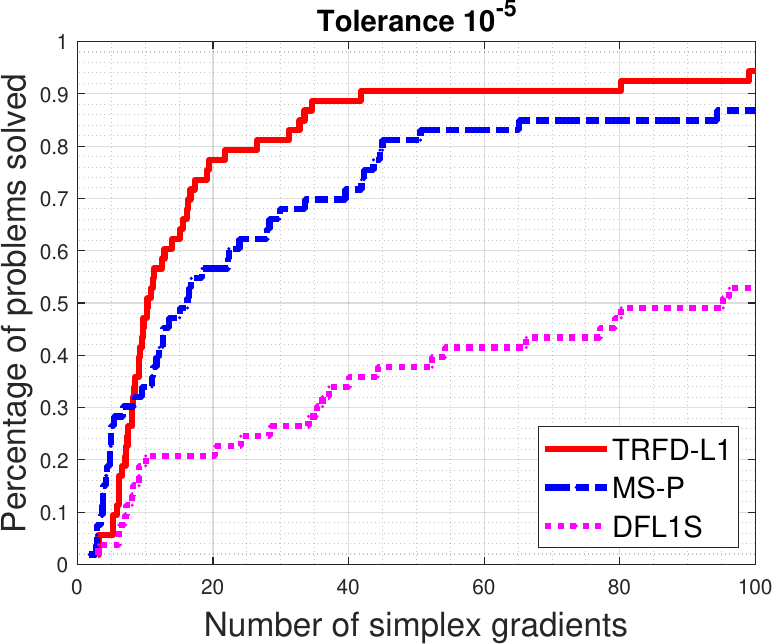}}
    \hspace{5mm}
    \subfigure
    {\includegraphics[width=0.33\textwidth]{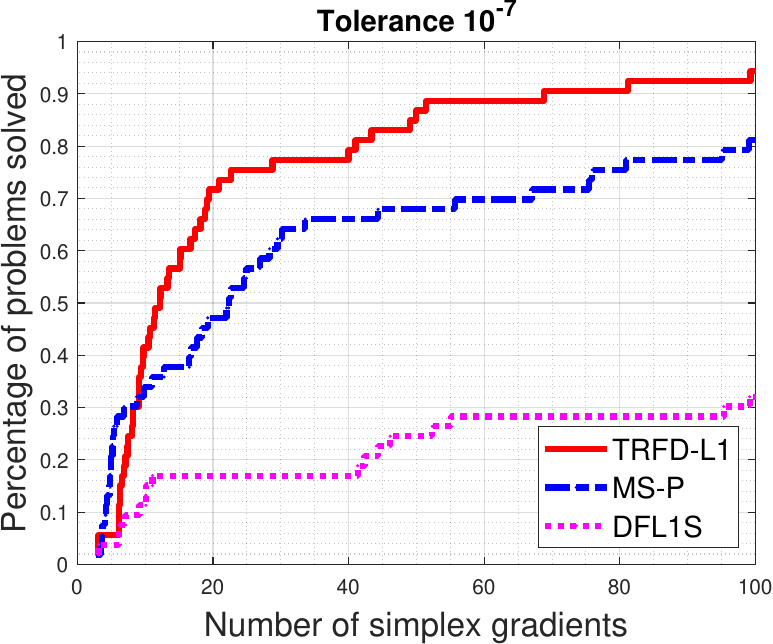}}
    \caption{Data profiles of TRFD-L1, MS-P and DFL1S on L1 problems}
    \label{fig:l1_results}
\end{figure}

\subsection{Minimax problems}\label{sub:minimax}

We also considered problems of the form $$\min_{x\in \mathbb{R}^n} \max_{i=1,...,m}\{F_i(x)\}.$$ We tested 43 functions $F:\mathbb{R}^n \to \mathbb{R}^m$ defined by Lukšan and Vlcek \cite{lukvsan2000test} and Di Pillo et al. \cite{di1993smooth}, for which $2\leq n \leq 50$ and $2\leq m\leq 130$. On these problems, the following codes were compared:
\\[0.15cm]
- \textbf{TRFD-M} Implementation of TRFD with $p=1$ if $\sqrt{m}<n$, and $p=\infty$ if $\sqrt{m}\geq n$. Parameters are the same used in TRFD-L1, with constants $L_{h,1}=1$, $L_{h,\infty}=1$, $c_{1,2}(m)=\sqrt{m}$, $c_{2,1}(n)=1$, $c_{\infty,2}(m)=1$ and $c_{2,\infty}(n)=\sqrt{n}$. Subproblems are solved using the MATLAB function \textit{linprog.m}.
\\[0.15cm]
- \textbf{TRFD-M2} Implementation of TRFD with $p=2$. Parameters are the same used in TRFD-L1, with constants $L_{h,2}=1$, $c_{2,2}(m)=1$ and $c_{2,2}(n)=1$. Subproblems are solved using the MATLAB function \textit{fmincon.m}.
\\[0.15cm]
- \textbf{MS-P}: Implementation of Manifold Sampling Primal \cite{LM2}, with the outer function $h(\,\cdot\,)$ provided by the file \textit{pw\_maximum}. 
\\[0.15cm]
- \textbf{DFMS}: Implementation described in Section 7 of \cite{grapiglia2016derivative}.
\\[0.15cm]
Figure \ref{fig:minimax_results} presents the data profiles comparing TRFD-M, MS-P and DFMS. As shown, TRFD-M and MS-P exhibited similar performances and both outperformed DFMS across all tolerances considered.
\vspace{0.15cm}
\begin{figure}[h!]
    \centering
    \subfigure
    {\includegraphics[width=0.33\textwidth]{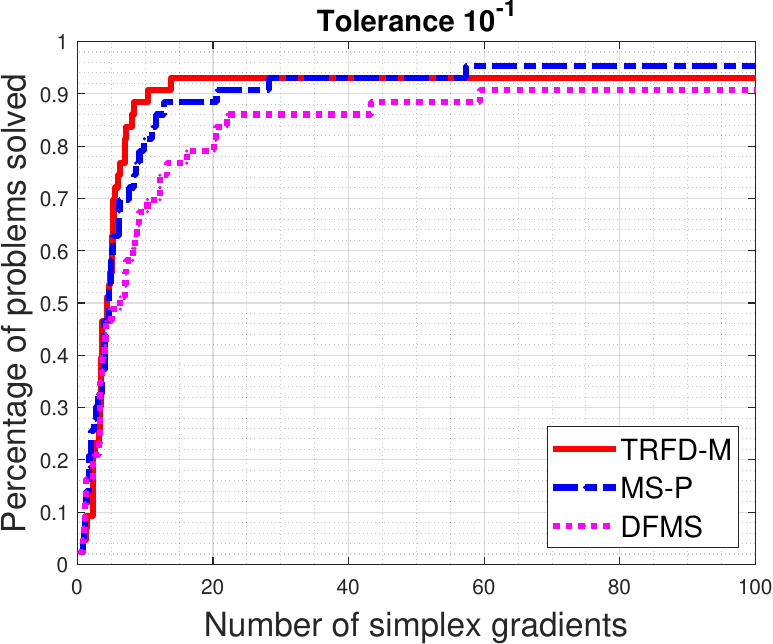}}
    \hspace{5mm}
    \subfigure
    {\includegraphics[width=0.33\textwidth]{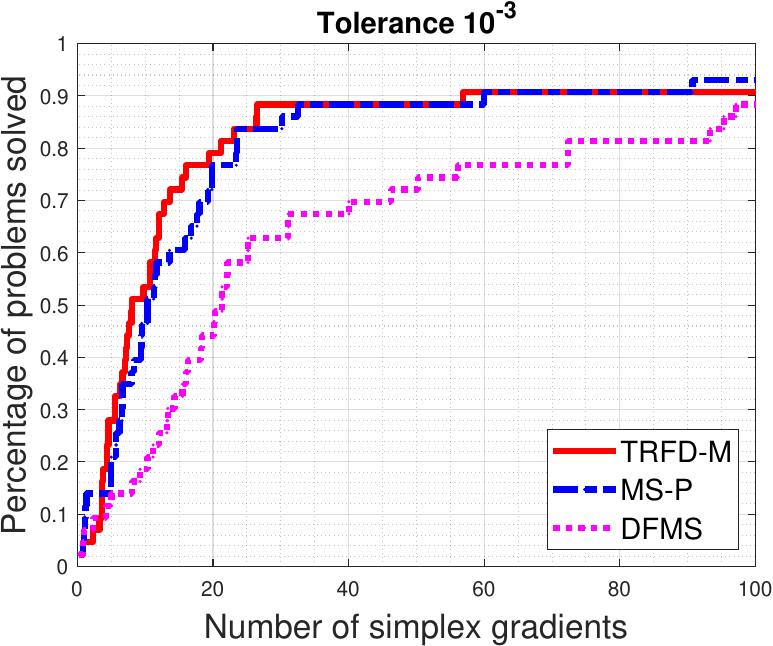}} 
    \subfigure
    {\includegraphics[width=0.33\textwidth]{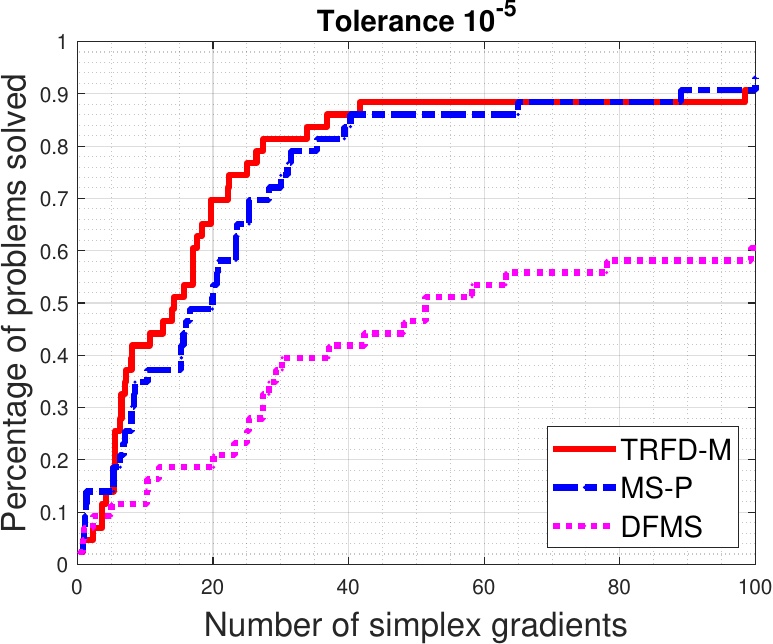}}
    \hspace{5mm}
    \subfigure
    {\includegraphics[width=0.33\textwidth]{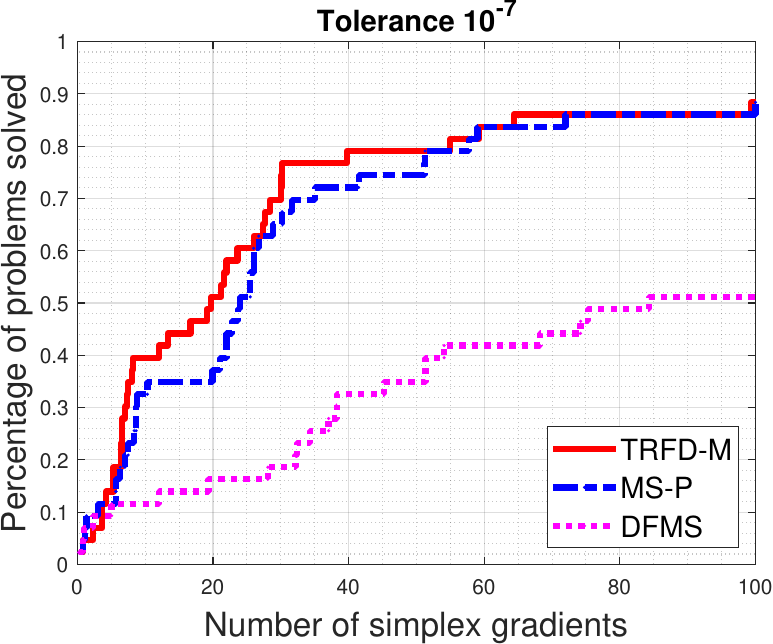}}
    \caption{Data profiles of TRFD-M, MS-P and DFMS on Minimax problems}
    \label{fig:minimax_results}
    \vspace{-3mm}
\end{figure}
\newpage
We also compared TRFD-M2 against MS-P and TRFD-M. The data profiles are shown in Figure \ref{fig:minimax_results_p=2}. For tolerances $10^{-3}$ and $10^{-5}$, TRFD-M2 performed slightly better than MS-P and TRFD-M. However, for tolerance $10^{-7}$, both MS-P and TRFD-M outperformed TRFD-M2.
\vspace{0.2cm}
\begin{figure}[h!]
    \centering
    \subfigure
    {\includegraphics[width=0.33\textwidth]{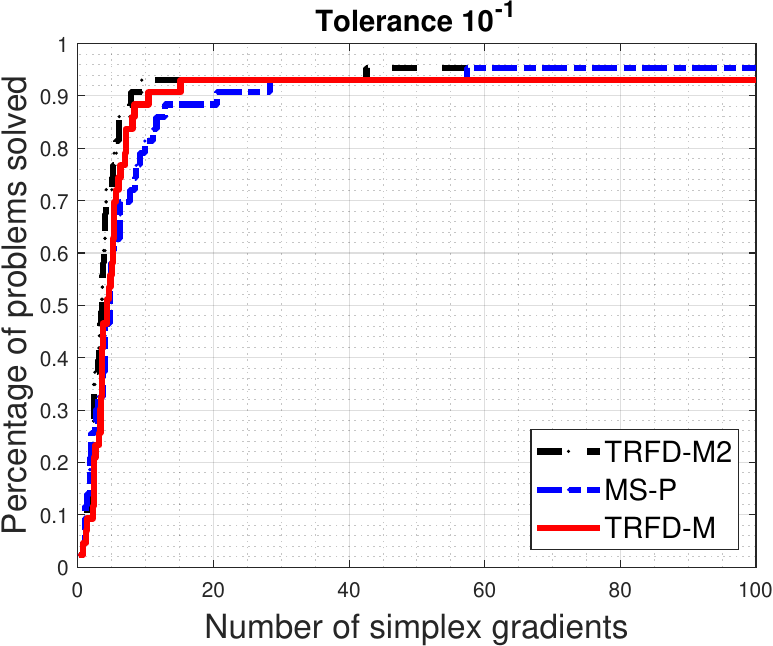}}
    \hspace{5mm}
    \subfigure
    {\includegraphics[width=0.33\textwidth]{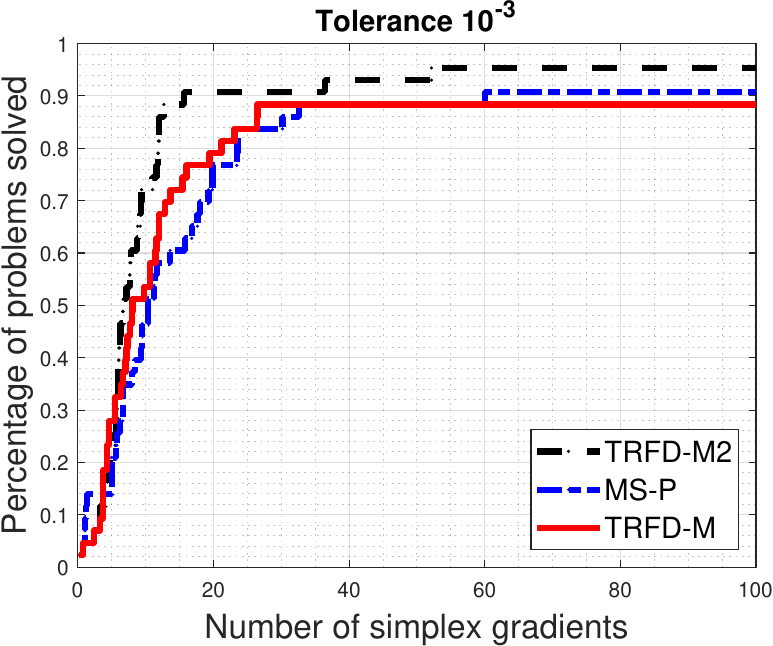}} 
    \subfigure
    {\includegraphics[width=0.33\textwidth]{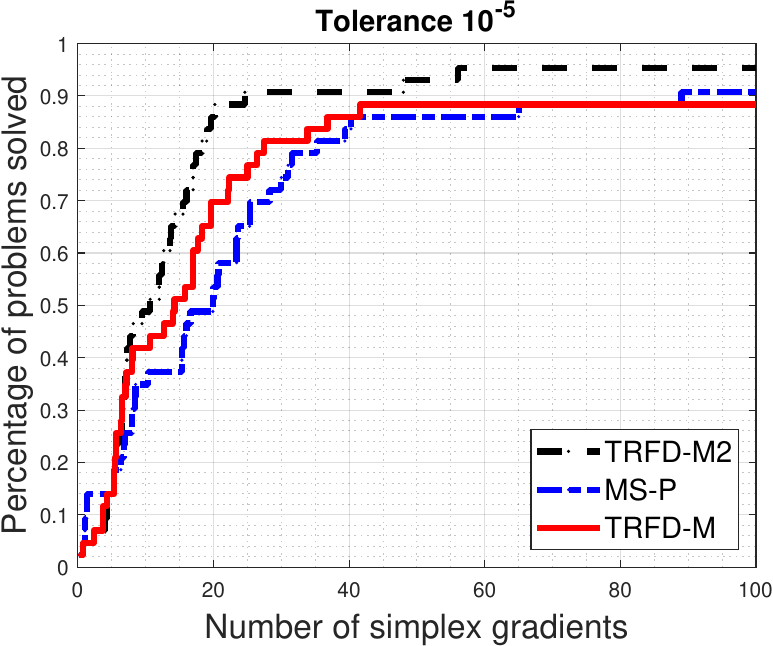}}
    \hspace{5mm}
    \subfigure
    {\includegraphics[width=0.33\textwidth]{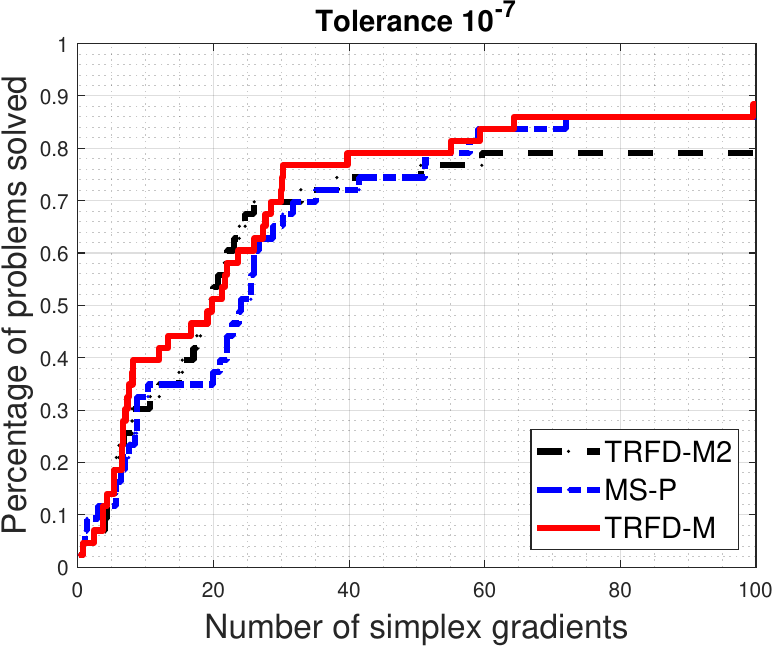}}
    \caption{Data profiles of TRFD-M2, MS-P and TRFD-M on Minimax problems}
    \label{fig:minimax_results_p=2}
\end{figure}
\section{Conclusion}\label{sec:conclusions}
In this paper, we introduced TRFD, a derivative-free trust-region method for minimizing composite functions of the form $f(x)=h(F(x))$ over a convex set $\Omega$. In the proposed method, trial points are obtained by minimizing models of the form $h(M_{k}(x_{k}+d))$ subject to the constraints $\|d\|_{p}\leq\Delta_{k}$ and $x_{k}+d\in\Omega$. Unlike existing model-based derivative-free methods for composite nonsmooth optimization, in which $M_{k}(x_{k}+d)$ is built as a linear or quadratic interpolation model of $F$ around $x_{k}$, TRFD employs $M_{k}(x_{k}+d)=F(x_{k})+A_{k}d$, where $A_{k}$ is an approximation for the Jacobian of $F$ at $x_{k}$, constructed using finite differences defined by a stepsize $\tau_{k}$. Special rules for updating $\tau_{k}$ and $\Delta_{k}$ allowed us to establish improved evaluation complexity bounds for TRFD in the nonconvex case. In particular, for L1 and Minimax problems, we proved that TRFD with $p=1$ and $p=2$ requires no more than $O(n \epsilon^{-2})$ evaluations of $F(\,\cdot\,)$ to find an $\epsilon$-approximate stationary point. Moreover, under the assumptions that $h(\,\cdot\,)$ is monotone and that the components of $F(\,\cdot\,)$ are convex, we established a complexity bound for the number of evaluations of $F(\,\cdot\,)$ that TRFD requires to find an $\epsilon$-approximate minimizer of $f(\,\cdot\,)$ on $\Omega$. For Minimax problems, our bound reduces to $O(n \epsilon^{-1})$ when we use $p=1$ or $p=2$ in TRFD. We concluded by presenting numerical results comparing implementations of TRFD against two model-based derivative-free trust-region methods, namely, Manifold Sampling \cite{LM2} and the derivative-free method from \cite{grapiglia2016derivative}. For L1 problems, TRFD outperformed the other two solvers, while for Minimax problems, TRFD demonstrated a competitive performance with Manifold Sampling.

\section*{Acknowledgments}
We are grateful to Jeffrey Larson for interesting discussions about composite optimization and for his assistance with the MS-P code.

\bibliographystyle{siamplain}

\end{document}